\documentclass[11pt]{amsart}
\usepackage{amsmath}
\usepackage{amsfonts}
\usepackage{amssymb}
\usepackage{graphicx}
\usepackage{pgfplots} %plotting 
\usepackage[centering, margin=1in]{geometry} %margins
\usepackage{comment}
\usepackage{tikz-cd} %commutative diagrams
\usepackage[abbrev,nobysame,alphabetic]{amsrefs} %references

\usepackage{pgf}
\pgfplotsset{compat=1.14}
\usetikzlibrary {positioning}
\usetikzlibrary {arrows,automata}
\usetikzlibrary{shapes}

\theoremstyle{plain}
\newtheorem{theorem}{Theorem}[section]
\newtheorem{proposition}[theorem]{Proposition}
\newtheorem{lemma}[theorem]{Lemma}
\newtheorem{corollary}[theorem]{Corollary}

\theoremstyle{definition}
\newtheorem{definition}[theorem]{Definition}

\newtheorem{notation}[theorem]{Notation}
\newtheorem{remark}[theorem]{Remark}

\theoremstyle{remark}
\newtheorem{example}[theorem]{Example}
\newtheorem*{claim}{Claim}
\newtheorem{step}{Step}

\numberwithin{equation}{section}

\DeclareMathOperator{\spec}{Spec}
\DeclareMathOperator{\proj}{Proj}
\DeclareMathOperator{\git}{/\!\!/}
\DeclareMathOperator{\im}{Im}
\DeclareMathOperator{\Hom}{Hom}
\DeclareMathOperator{\Pic}{Pic}
\DeclareMathOperator{\Cl}{Cl}
\DeclareMathOperator{\codim}{codim}
\DeclareMathOperator{\Id}{Id}
\DeclareMathOperator{\weight}{weight}
\DeclareMathOperator{\rk}{rk}
\DeclareMathOperator{\Gr}{Gr}

\begin{document}

\title{Bott vanishing using GIT and quantization}
\author{Sebasti\'an Torres}

\address{Department of Mathematics\\University of Miami\\1365 Memorial Drive\\Ungar 503\\Coral Gables, FL 33146}
\email{sdtorres@miami.edu}

\begin{abstract}
A smooth projective variety $Y$ is said to satisfy Bott vanishing if $\Omega_Y^j\otimes L$ has no higher cohomology for every $j$ and every ample line bundle $L$. Few examples are known to satisfy this property. Among them are toric varieties, as well as the quintic del Pezzo surface, recently shown by Totaro. Here we present a new class of varieties satisfying Bott vanishing, namely stable GIT quotients of $(\mathbb{P}^1)^n$ by the action of $PGL_2$, over an algebraically closed field of characteristic zero. For this, we use the work done by Halpern-Leistner on the derived category of a GIT quotient, and his version of the quantization theorem. We also see that, using similar techniques, we can recover Bott vanishing for the toric case.
\end{abstract}

\maketitle 

\section{Introduction}\label{introduction section} 

We say that a smooth projective variety $Y$ satisfies Bott vanishing if for every ample line bundle $L$ we have
\begin{align} \label{vanishing}
H^i(Y,\Omega_Y^j \otimes L)=0 \quad \forall i>0, \forall j\geq0.
\end{align}

In \cite{totaro}, Totaro gives a geometric interpretation of what it means for a K3 surface to have this property. In general, it is not clear what the geometric meaning of Bott vanishing is, although it is certainly useful, when it holds, to compute sections of some vector bundles.  

This property turns out to be very restrictive. For instance, a Fano variety that satisfies Bott vanishing 
must be rigid, and even among rigid Fano varieties, the property is known to fail for quadrics of dimension at least $3$ and Grassmannians other than $\mathbb{P}^n$ (see the discussion in \cite{totaro} and the references therein). Smooth toric varieties are among the few known examples of varieties satisfying Bott vanishing. Several different proofs can be found in \cite{cox}, \cite{toric-buch}, \cite{toric-fujino}, \cite{toric-mustata}. In fact, vanishing (\ref{vanishing}) is shown for any projective toric variety, where $\Omega^j_Y$ is defined as the pushforward of $\Omega^j_{Y^0}$ from the smooth locus $Y^0$ (see e.g. \cite{toric-fujino}). Up until Totaro's paper \cite{totaro}, there were no known non-toric examples of rationally connected varieties with this property. He proves that the quintic del Pezzo surface over any field satisfies Bott vanishing, as well as coming up with several other non-toric examples from K3 surfaces. Namely, he proves that Bott vanishing fails for K3 surfaces of degree less than $20$ or equal to $22$, while it holds for all K3 surfaces of degree $20$ or at least $24$ with Picard number $1$.

The present work was motivated by \cite{totaro} and it continues the quest for non-toric examples of varieties satisfying Bott vanishing. Observe that the quintic del Pezzo surface can be realized as a GIT quotient of $(\mathbb{P}^1)^5$ by the diagonal action of $PGL_2$ with respect to the symmetric polarization $\mathcal{L}=\mathcal{O}(1,\ldots,1)$. We prove that in fact Bott vanishing holds for every stable GIT quotient $(\mathbb{P}^1)^n\git_\mathcal{L}PGL_2$, over an algebraically closed field of characteristic $0$. In particular, this gives one new Fano example for each even dimension.

\begin{theorem}\label{main theorem}
Let $PGL_2$ act diagonally on $(\mathbb{P}^1)^n$, over an algebraically closed field of characteristic zero. Suppose $\mathcal{L}$ is a $PGL_2$-linearized ample line bundle on $(\mathbb{P}^1)^n$ admitting no strictly semi-stable locus. Then the GIT quotient $Y=(\mathbb{P}^1)^n\git_\mathcal{L}PGL_2$ satisfies Bott vanishing.
\end{theorem}

To prove this, we use the results of Halpern-Leistner's to carry out computations in the derived category of GIT quotients \cite{dhl}. Namely, we use his version of Teleman's Quantization Theorem \cites{teleman, dhl}. Roughly speaking, this theorem allows us, under certain conditions, to compute cohomologies $H^\cdot(X\git G,F)$ on the GIT quotient as $G$-equivariant cohomologies $H_G^\cdot(X,\tilde{F})$ on the ambient quotient stack $[X/G]$, where $\tilde{F}$ must be some object in the derived category of $[X/G]$ descending to $F$. The conditions required have to do with the weights of $\tilde{F}$ over the unstable locus, and have to be tested on a Kempf-Ness stratification of it. We refer the reader to \cites{castravet1,castravet1.1,castravet2} for a complete description of the GIT of the action of $PGL_2$ on $(\mathbb{P}^1)^n$ and a stratification of the unstable locus, as well as a description of the derived category of the quotient stack $[(\mathbb{P}^1)^n/PGL_2]$ in terms of an equivariant full exceptional collection.

For our case, in Section \ref{git section} we find an object $\Lambda^j L_{\mathfrak{X}}\otimes \mathcal{L}$ descending to $\Omega^j_{Y} \otimes L$ in the GIT quotient. Explicitly, $\Lambda^j L_{\mathfrak{X}}$ is the $j$-th exterior power of the two-step complex $\Omega_X \rightarrow \mathfrak{g}^\vee$ determined by the action, where $\mathfrak{g}$ is the Lie algebra of the group. In Section \ref{weights section} we check that this object satisfies the weights condition from the quantization theorem, and then devote most of the work to the corresponding computation of cohomology in the ambient quotient stack. We first see that, as a consequence of the Bott vanishing property on $X=(\mathbb{P}^1)^n$, this amounts to computing cohomologies of the complex of invariant global sections of the object $\Lambda^j L_{\mathfrak{X}}\otimes \mathcal{L}$ on $(\mathbb{P}^1)^n$ (see Lemma \ref{F computes hypercohomology}). Following the spirit of Weyman's method of geometric syzygies \cite{weyman}, we view these as sections of some sheaves in the product $X\times \mathbb{P}(\mathfrak{g})$, rather than sheaves on $X$. Let $M\subset X\times\mathbb{P}(\mathfrak{g})$ be the scheme-theoretic zero locus of the section $s:\Omega_X\rightarrow \mathfrak{g}^\vee$. Koszul resolution of $M$, together with an associated spectral sequence, yields then vanishing for the $i$-th cohomology in (\ref{vanishing}), for $i\neq 0,j$. This is discussed in Section \ref{koszul section}. The techniques used up to this point do not require the particular context of $PGL_2$ acting on $(\mathbb{P}^1)^n$, and can be applied to other GIT quotients $X\git G$ satisfying certain hypotheses. The main properties that we need are that of $X$ itself satisfying Bott vanishing and $M$ being a local complete intersection.

Next, we observe in Section \ref{abelian section} that in the case of an abelian group acting on a smooth affine variety, very similar techniques can be used to get a stronger vanishing result (see Theorem \ref{theorem bott for abelian}).

\begin{theorem}
Let $G$ be an abelian reductive group acting on a smooth affine variety $X$, over an algebraically closed field of characteristic zero. Let $\mathcal{L}$ be a linearization with no strictly semi-stable locus and descending to a line bundle $L$ in the GIT quotient $Y=X\git_\mathcal{L}G$. Suppose $G$ acts freely on $X^{ss}$ except for a subset of codimension at least $2$. Then $H^i(Y,\Omega_Y^j \otimes L)=0 \; \forall i>0, \forall j\geq0$.
\end{theorem}

Observe that this is not the same as Bott vanishing, since the formula is only stated for the descent of the linearization $\mathcal{L}$, while Bott vanishing requires (\ref{vanishing}) to hold for any ample line bundle. However, this vanishing is essentially all that needs to be verified in the particular case that $X=\mathbb{A}^d$, where we have an explicit description of the $G$-ample cone and the ring of invariants, as detailed in \cite{hu-keel}. As a consequence, we obtain yet another proof of Bott vanishing for the toric case in characteristic zero (see Theorem \ref{thm bott for toric}). More precisely, we show it for a $\mathbb{Q}$-factorial projective toric variety over an algebraically closed field of characteristic zero, using its description as a GIT quotient of the affine space due to Cox \cite{cox-coxring}. We then hope that these techniques, using Halpern-Leistner's results, may be used to yield more examples of varieties satisfying Bott vanishing.

In Section \ref{main section} we finish the proof of Theorem \ref{main theorem}. Here we mostly deal with the vanishing of $H^j(Y,\Omega^j_Y\otimes L)$, where $Y=(\mathbb{P}^1)^n\git_\mathcal{L} PGL_2$. Given the work developed in the previous sections, this amounts to computing cohomology of the complex of invariant global sections of the object $\Lambda^j L_{\mathfrak{X}}\otimes \mathcal{L}$ defined in Section \ref{git section}. More precisely, we are left with the computation of the last cohomology of this complex, which is the same as investigating surjectivity of the map of invariant sections $H^0(X,\Omega_X\otimes S^{j-1}\mathfrak{g}^\vee)^{G}\rightarrow H^0(X,S^j\mathfrak{g}^\vee)^G$. To do this we use techniques that are particular to our case, that is, $PGL_2$ acting on $X=(\mathbb{P}^1)^n$. Namely, we handle invariant sections using the description of \cite{vakil+3}, \cite{vakil+3-II}, where sections are identified with linear combinations of directed graphs with prescribed degrees on the vertices.

%\begin{acknowledgements}
\subsection{Acknowledgments}
I would like to express my gratitude to my PhD advisor, Jenia Tevelev, for his guidance and insightful comments throughout this work. I also want to thank David Cox and Ana-Maria Castravet for helpful discussions, and Burt Totaro for useful comments and feedback. This project has been partially supported by the NSF grant DMS-1701704 (PI Jenia Tevelev).
%\end{acknowledgements}

\section{GIT quotients and quotient stacks}\label{git section}

We will consider a smooth projective-over-affine variety $X$ over an algebraically closed field $\Bbbk$ of characteristic $0$, meaning a closed subvariety of $\mathbb{A}^r\times\mathbb{P}^d$, with a reductive group $G$ acting on $X$. For an ample $G$-linearized line bundle $\mathcal{L}$ on $X$, we call $X^{ss}$ and $X^{s}$ the semi-stable and stable loci, respectively. We write $X^{us}$ for the unstable locus $X - X^{ss}$. Denote by $Y=X \git_\mathcal{L} G$ the corresponding GIT quotient, and $\pi : X^{ss} \twoheadrightarrow Y$ the quotient map from the semi-stable locus. The map $\pi$ is affine, in particular $\pi_*$ is exact, and we have $\pi_*(\mathcal{O}_{X^{ss}})^G=\mathcal{O}_Y$. The restriction to the stable locus gives a geometric quotient $X^s\rightarrow X^s\git G$. We will be mostly interested in the cases where there is no strictly semi-stable locus, that is, $X^{ss}=X^s$.

Given the action of $G$ on $X$, we denote by $\mathfrak{X}$ the corresponding quotient stack $[X/G]$. Coherent $\mathcal{O}_{\mathfrak{X}}$-modules are $G$-equivariant coherent $\mathcal{O}_X$-modules, and indeed $D^b(\mathfrak{X})=D^b_G(X)$, that is, an object in $D^b(\mathfrak{X})$ is represented by a $G$-equivariant chain complex in $D^b(X)$. Cohomology in $\mathfrak{X}$ is $G$-equivariant cohomology on $X$. For a given $G$-linearized ample line bundle $\mathcal{L}$, denote by $\mathfrak{X}^{ss}$ the corresponding open substack $[X^{ss}/G]$, with its inclusion $\imath:\mathfrak{X}^{ss} \hookrightarrow \mathfrak{X}$. The quotient map $\pi$ gives a map from the quotient stack $p:\mathfrak{X}^{ss}\rightarrow X\git G$. If $X^{ss}=X^s$, $\mathfrak{X}^{ss}$ is a Deligne-Mumford stack (see e.g. \cite{teleman} for a discussion), and the GIT quotient $Y=X\git_\mathcal{L} G$ is a coarse moduli space for $\mathfrak{X}^{ss}$. In this case, the map $p$ is finite. If, further, the action is free on $X^{ss}$, then $p$ is an isomorphism.

\begin{notation}
We denote by $H^i$ the $i$-th hypercohomology of a complex in $D^b(X)$, that is, the $i$-th derived functor $R^i\Gamma$ of the functor of global sections. We denote by $\mathcal{H}^i$ the $i$-th cohomology of the complex itself. Also, denote by $H_G^i$ the derived functor of invariant global sections $\Gamma_G$.
\end{notation}

\begin{remark}
$G$ is a reductive group, so taking $G$-invariants is an exact functor on finite dimensional representations. Therefore, for a complex $F^\cdot\in D^b(\mathfrak{X})$, $H^i(\mathfrak{X},F^\cdot)=H_G^i(X,F^\cdot)=H^i(X,F^\cdot)^G$.
\end{remark}

For an object $\tilde{F} \in D^b(\mathfrak{X}^{ss})$, we say that it ``descends'' to $F \in D^b({X}\git G)$ if $p^*(F)\cong\tilde{F}$, that is, if there is a $G$-equivariant isomorphism $\pi^*(F)\cong\tilde{F}$, where $p^*$, $\pi^*$ denote the derived pullbacks. Observe that given $\tilde{F}$, such $F$ is unique up to isomorphism: it has to be the pushforward $p_*(\tilde{F})=\pi_*(\tilde{F})^G$. In the case that $G$ acts freely on $X^{ss}$, $p$ is an isomorphism, so the categories $D^b(\mathfrak{X}^{ss})$ and $D^b(X\git G)$ are equivalent, via $F\mapsto \pi_*(F)^G$. In general, for an object $\tilde{F} \in D^b(\mathfrak{X})$, we say that it descends to $F \in D^b({X}\git G)$ if its restriction $\imath^*(\tilde{F})$ does.

\subsection{Kempf-Ness stratifications}

Let $\lambda:\mathbb{G}_m \rightarrow G$ be a one-parameter subgroup. If $F$ is a $G$-linearized line bundle on $X$ and $y \in X^\lambda$ is a $\lambda$-fixed point, $\lambda$ acts in the fiber $F_y$, with a given weight which we denote $\weight_\lambda F_y$. Similarly, if $F$ is a $G$-equivariant vector bundle, its $\lambda$-weights on $F_y$ are the eigenvalues of the action of $\lambda$ on $F_y$. For a $G$-equivariant complex $F^\cdot \in D^b(X)$, we refer to the $\lambda$-weights of $\mathcal{H}^i(F^\cdot)$ for all $i$ as the $\lambda$-weights of $F^\cdot$.

Suppose we have a $G$-linearized ample line bundle $\mathcal{L}$ for the action of $G$ on $X$. The unstable locus $X^{us}=X-X^{ss}$ can be described using the Hilbert-Mumford numerical criterion.

\begin{theorem}[(Hilbert-Mumford criterion)]
$X^{ss}$ (resp. $X^s$) consists of the points $x$ such that $\weight_\lambda \mathcal{L}_y \geq 0$ (resp. $>0$) for all $\lambda$ such that $y=\lim_{t \rightarrow 0} \lambda(t) x$ exists.
\end{theorem}

Using the Hilbert-Mumford criterion, one can define what is called a Kempf-Ness (KN) stratification of the unstable locus, as described below (see \cite[\S2.1]{dhl} for a more detailed discussion). For a given one-parameter subgroup $\lambda: \mathbb{G}_m \rightarrow G$, and a connected component $Z$ of the fixed locus $X^\lambda$ we define the \textit{blade} of $\lambda,Z$ as
\begin{align*}
Y_{Z,\lambda} = \{ x \in X \mid \lim_{t \rightarrow 0} \lambda (t) \cdot x \in Z \},
\end{align*}
i.e., the points that are attracted to $Z$ as $t \rightarrow 0$. Under our assumption of smoothness on $X$, both $Z$ and $Y_{Z,\lambda}$ are smooth varieties, and in fact the projection $Y_{Z,\lambda}\to Z$, $x\mapsto \lim_{t \rightarrow 0} \lambda (t) \cdot x$ realizes $Y_{\lambda,Z}$ as a locally trivial vector bundle over $Z$. Now define $\mu(\lambda,Z)=-\frac{1}{|\lambda|}\weight_\lambda \mathcal{L}|_Z$, where $|\lambda|$ is a norm over one-parameter subgroups given by a choice of some suitable inner product in the cocharacter lattice of a maximal torus of $G$. Then we can write a stratification of the unstable locus by iteratively selecting a pair $(Z_\alpha,\lambda_\alpha)$ such that $\mu$ is positive and maximal among those $(Z,\lambda)$ for which $Z$ is not contained in the previously defined strata. We may assume $\lambda$ is a one-parameter subgroup of a maximal torus. Let $Z_\alpha^\circ \subset Z_\alpha$ be the open subset not intersecting any previous strata. We call $Y_\alpha = Y_{Z_\alpha^\circ,\lambda}$, the set attracted to $Z_\alpha^\circ$. Then the next stratum is $S_\alpha=G\cdot Y_\alpha$. It can be proved that this leads to an ascending sequence of finitely many $G$-equivariant open subvarieties $X^{ss}=X_0 \subset X_1 \subset \cdots \subset X$, where each $X_\alpha \backslash X_{\alpha-1}=S_\alpha$ is one of these strata.

For each stratum $S_\alpha$, given by a pair $\lambda_\alpha,Z_\alpha$, one can define the Levi subgroup $L_\alpha \subset G$ given by elements $g\in G$ that centralize $\lambda_\alpha$ and satisfy $g(Z_\alpha)\subset Z_\alpha$; and the parabolic subgroup $P_\alpha$ as the $g \in G$ such that $\lim_{t\rightarrow 0} \lambda (t) g \lambda(t)^{-1}$ exists and is in $L_\alpha$. We have a short exact sequence
\begin{align*}
1 \rightarrow U_\alpha \rightarrow P_\alpha \rightarrow L_\alpha \rightarrow 1
\end{align*}
where $U_\alpha =\{g \in G \mid \lim_{t\rightarrow 0} \lambda (t) g \lambda(t)^{-1} =1 \}$. The inclusion $L_\alpha \hookrightarrow P_\alpha$ allows us to write $P_\alpha$ as a semidirect product $U_\alpha \rtimes L_\alpha$.

We have that the action map $G\times_{P_\alpha}Y_\alpha\rightarrow G\cdot Y_\alpha$ is an isomorphism, where $G\times_{P_\alpha}Y_\alpha\overset{\pi}\rightarrow G/P_\alpha$ is the fibered bundle with fiber isomorphic to $Y_\alpha$. We also know that the $\lambda_\alpha$-weights of the conormal bundle $\mathcal{N}^\vee_{S_\alpha}$ restricted to $Z_\alpha$ are positive. Also, it is not hard to see that the $\lambda_\alpha$-weights of $T_{Y_\alpha}|_{Z_\alpha}$ are nonnegative (see \cite[\S 1.3]{vgit}, \cite[\S 12-13]{kirwan} for details).

\subsection{The complex $L_{\mathfrak{X}}$}

By differentiating the action of $G$ on $X$, we get a $G$-equivariant morphism of vector bundles $s: \mathfrak{g}\otimes \mathcal{O}_X \rightarrow T_X$. This map can be viewed as a $G$-equivariant vector field $s\in H^0(X,T_X\otimes \mathfrak{g}^\vee)^G$. By abuse of notation, we will write $\mathfrak{g}$ for $\mathfrak{g}\otimes \mathcal{O}$. Taking the dual we get a complex $L_\mathfrak{X}$:

\begin{notation}
The cotangent complex $L_\mathfrak{X} \in D^b(\mathfrak{X})$ is defined as the two-step $G$-linearized complex $\Omega_X \rightarrow \mathfrak{g}^\vee$, in degrees $0$ and $1$. We denote by $\Lambda^j L_\mathfrak{X}$ the $j$-th (derived) exterior power of this complex (see \cite[\S 2.4]{weyman}, \cite[\S I.4]{illusie1}).
\end{notation}

The complex $\Lambda^j L_\mathfrak{X}$ can then be written as the Koszul complex $0\rightarrow\Omega^j_{X} \rightarrow \Omega^{j-1}_{X} \otimes \mathfrak{g}^\vee \rightarrow \cdots \rightarrow S^j \mathfrak{g}^\vee \rightarrow 0$, in degrees $0$ to $j$, associated to $s:\mathfrak{g}\rightarrow T_X$ (see \cite{totaro-classif} and the references therein). 

Denote by $L_{\mathfrak{X}^{ss}}$ the restriction of $L_{\mathfrak{X}}$ to $\mathfrak{X}^{ss}$. Observe that the restriction of the map $\mathfrak{g}\rightarrow T_{X}$ to the stable locus is injective, since stabilizers are finite in $X^s$. This implies that if $X^{ss}=X^s$, we have a surjection $\Omega_{X^{ss}}\twoheadrightarrow \mathfrak{g}^\vee\otimes\mathcal{O}_{X^{ss}}=\Omega_{X^{ss}/\mathfrak{X}^{ss}}$, the relative cotangent bundle, and there is a short exact sequence $0 \rightarrow q^*\Omega_{\mathfrak{X}^{ss}} \rightarrow \Omega_{X^{ss}}\rightarrow \mathfrak{g}^\vee\rightarrow 0$, where $q:X^{ss}\rightarrow \mathfrak{X}^{ss}$ is the canonical quotient map to the quotient stack. In particular, $L_{\mathfrak{X}^{ss}}$ is (isomorphic to) a vector bundle, and the same is true for its exterior powers. If, further, the action is free on ${X}^{ss}$, the GIT quotient $Y$ is isomorphic to $\mathfrak{X}^{ss}$ and we have a $G$-equivariant short exact sequence of vector bundles $0 \rightarrow \pi^*\Omega_Y \rightarrow \Omega_{X^{ss}}\rightarrow \mathfrak{g}^\vee\rightarrow 0$. From this we see that in this case the restriction $L_{\mathfrak{X}^{ss}}$ is isomorphic to $\pi^*\Omega_{Y}$ in $D^b(\mathfrak{X}^{ss})$, that is, $L_{\mathfrak{X}}$ descends to $\Omega_{Y}$ and, for the same reason, each $\Lambda^j L_\mathfrak{X}$ descends to $\Omega^j_{Y}$.

\section{Weights and cohomology}\label{weights section}

The Quantization Theorem states that cohomologies of a complex $F \in D^b(\mathfrak{X}^{ss})$ can be computed in $\mathfrak{X}$ if $F$ is the restriction of some complex in $D^b(\mathfrak{X})$ satisfying a numerical condition related to the $\lambda$-weights in the unstable strata. It was proved by Teleman \cite{teleman} in the case where $F$ is a vector bundle, and Halpern-Leistner \cite[Theorem 3.29]{dhl} proved it for an arbitrary object in the derived category.

\begin{theorem}[(Quantization Theorem)] \label{teleman}
Let $\{S_\alpha\}$ be a KN stratification of the unstable locus, with the corresponding one-parameter subgroups $\lambda_\alpha$ and connected components $Z_\alpha$ of the fixed locus $X^{\lambda_\alpha}$. Let 
\begin{align*}
\eta_\alpha=\weight_{\lambda_\alpha}(\det \mathcal{N}^\vee_{S_\alpha})|_{Z_\alpha}.
\end{align*}
Suppose $\tilde{F} \in D^b(\mathfrak{X})$ restricts to $F \in D^b(\mathfrak{X}^{ss})$. If $\tilde{F}$ satisfies that, for every $l$ and every $\alpha$, all the $\lambda_\alpha$-weights of $\mathcal{H}^l(\tilde{F})|_{Z_\alpha}$ are $< \eta_{\alpha}$, then $H^\cdot(\mathfrak{X}^{ss},F)=H^\cdot(\mathfrak{X},\tilde{F})$.
\end{theorem}

\begin{remark}
In particular, if the action is free on $X^{ss}$, this computes the cohomologies $H^\cdot(Y,F)$, where $Y$ is the GIT quotient.
\end{remark}

In fact, Halpern-Leistner's work says much more. His Categorical Kirwan Surjectivity says that, given a Kempf-Ness stratification $\{S_\alpha\}$ with inclusions $\sigma_\alpha:S_\alpha \hookrightarrow X$, we have that for any choice of integers $\{w_\alpha\}$, the full triangulated subcategory
\begin{align*}
G_w = \{F \in D^b(\mathfrak{X}) \mid \sigma_\alpha^*(F) \text{ is supported in weights } [w_\alpha,w_\alpha+\eta_\alpha)\}
\end{align*}
is equivalent to $D^b(\mathfrak{X}^{ss})$ via the restriction functor. For our purposes, we will only need to use the Quantization Theorem, which is related to fully-faithfulness of this restriction.

\begin{example}[\cite{dhl}]
Write $\mathbb{P}^n$ as the GIT quotient of $X=\mathbb{A}^{n+1}$ by $G=\mathbb{G}_m$. Write $\mathcal{O}_X(d)$ for the trivial line bundle on $\mathbb{A}^{n+1}$ with the linearization given by the character $t \mapsto t^d$, so that $\mathcal{O}_X(d)$ descends to $\mathcal{O}_{\mathbb{P}^n}(d)$ on $\mathbb{P}^n=\mathfrak{X}^{ss}$. The unstable locus is just the origin, and it is destabilized by $\lambda:t\mapsto t^{-1}$. We compute $\eta=n+1$ and $\weight_\lambda \mathcal{O}_X(d)=-d$. By the quantization theorem, $H^i(\mathbb{P}^n,\mathcal{O}_{\mathbb{P}^n}(d))=H_{\mathbb{G}_m}^i(\mathbb{A}^{n+1},\mathcal{O}_X(d))$ as long as $d>-n-1$, so $H^i(\mathbb{P}^n,\mathcal{O}_{\mathbb{P}^n}(d))=0$ whenever $i>0$ and $d>-n-1$. In fact, it turns out that $\mathcal{O}_{\mathbb{P}^n}(w),\ldots,\mathcal{O}_{\mathbb{P}^n}(w+n)$ is a full exceptional collection on $\mathbb{P}^n$ for any $w \in \mathbb{Z}$.
\end{example}

\begin{remark}
In some references, the blades $Y_{Z,\lambda}$ are defined by taking the limit as $t \rightarrow \infty$ rather than $0$. This causes a change of sign in the hypothesis of Theorem \ref{teleman}, namely, requiring that the weights be $>-\eta_\alpha$ instead. We will stick to the convention in \cite{dhl}, that is, the limit in $Y_{Z,\lambda}$ is taken as $t\rightarrow 0$ and the weights condition stays as stated above.
\end{remark}

In the following theorem and corollary, we check that we can apply the Quantization Theorem to $\Lambda^jL_\mathfrak{X}\otimes\mathcal{L}$, where $\mathcal{L}$ is the $G$-linearized ample line bundle on $X$. In the holomorphic setting, this was observed in \cite[Theorem 7.1]{teleman}.

\begin{theorem}
Let $G$ be a reductive group acting on $X$, and $\{S_\alpha\}$ a KN stratification of the unstable locus as described before. The $\lambda_\alpha$-weights of the complex $\Lambda^j L_{\mathfrak{X}}|_{Z_\alpha}$ are all $\leq \eta_\alpha$. If $G$ is abelian, then the $\lambda_\alpha$-weights of the individual terms $(\Lambda^j L_{\mathfrak{X}})^p|_{Z_\alpha}=(\Omega^{j-p}_X\otimes S^p\mathfrak{g}^\vee)|_{Z_\alpha}$ of the complex are all $\leq \eta_\alpha$.
\end{theorem}

\begin{proof}
Let $Z=Z_\alpha$ correspond to the stratum $S_\alpha$, and let $\lambda = \lambda_\alpha$ be the corresponding one-parameter subgroup. Since the weights condition is local, it is enough to compute them when we further restrict to an open affine $Z'\subset Z$. Consider the restriction
\begin{align} \label{restricted complex}
\mathfrak{g} \rightarrow T_X|_{Z'}
\end{align}
of the dual of $L_\mathfrak{X}$ to $Z'$. Include $\lambda$ in a maximal torus of $G$ and let $\mathfrak{h} \subset \mathfrak{g}$ be the corresponding Cartan subalgebra. We can write a root decomposition
\begin{align*}
\mathfrak{g}=\mathfrak{h} \oplus \bigoplus_{\beta \in \Phi} \mathfrak{g}_\beta.
\end{align*}
Observe that $\mathfrak{p} = \mathfrak{h} \oplus \bigoplus_{\beta(d\lambda) \geq 0} \mathfrak{g}_\beta$ is precisely the Lie algebra of the parabolic subgroup $P=P_\alpha$. Call $\mathfrak{n}^{-}=\bigoplus_{\beta(d\lambda) < 0} \mathfrak{g}_\beta$, so that $\mathfrak{g} = \mathfrak{p} \oplus \mathfrak{n}^{-}$.

Let $Y= Y_\alpha^\circ$ the corresponding blade. Using the isomorphism $G \times_P Y \cong G \cdot Y$, we get a short exact sequence of tangent sheaves:
\begin{align*}
0 \rightarrow G\times_P T_Y \rightarrow T_{G \cdot Y} \rightarrow \pi^*T_{G/P} \rightarrow 0.
\end{align*}
When we restrict to $Z'$, this sequence splits, since these are vector bundles and $Z'$ is affine. Therefore we have $T_{G\cdot Y} |_{Z'} = T_Y |_{Z'} \oplus \pi^*(T_{G/P})|_{Z'}$. The first summand has only nonnegative $\lambda$-weights, while the second summand has all negative $\lambda$-weights. Indeed, by the definition we have $\mathfrak{n}^{-} \overset{\sim}{\rightarrow} \pi^*(T_{G/P}) |_{Z'}$.

Now the restriction of the sequence
\begin{align*}
0 \rightarrow T_{G\cdot Y} \rightarrow T_X|_{G\cdot Y} \rightarrow \mathcal{N}_{G\cdot Y} X \rightarrow 0
\end{align*}
to $Z'$ splits, again because $Z'$ is affine, so we obtain $T_X |_{Z'} = T_Y|_{Z'} \oplus (\pi^*T_{G/P}) |_{Z'} \oplus (\mathcal{N}_{G\cdot Y}X)|_{Z'}$. Then the complex (\ref{restricted complex}) can be written as the direct sum of the complexes
\begin{align*}
\mathfrak{n}^{-} \overset{\sim}{\rightarrow} (\pi^*T_{G/P})|_{Z'} 
\end{align*}
and
\begin{align*}
\mathfrak{p} \rightarrow T_Y|_{Z'} \oplus (\mathcal{N}_{S}X)|_{Z'}.
\end{align*}
Similarly, the restricted complex $L_\mathfrak{X}|_{Z'}=\left[ \Omega_X|_{Z'}\rightarrow \mathfrak{g}^\vee\right]$ is written as a direct sum of two complexes, namely the duals of the complexes above. Therefore, by \cite[Proposition 2.4.7]{weyman}, the exterior powers of $\Omega_X|_{Z'}\rightarrow \mathfrak{g}^\vee$ are quasi-isomorphic to those of the complex $\Omega_Y|_{Z'} \oplus (\mathcal{N}^\vee_{S}X)|_{Z'} \rightarrow \mathfrak{p}^\vee$. Now the exterior and symmetric powers of $\mathfrak{p}$ and $T_Y|_{Z'}$ all have nonnegative $\lambda$-weights, so their duals have weights $\leq 0$. On the other hand, the weights of $\mathcal{N}^\vee_{S}X |_{Z'}$ are all positive and the sum of all of them is $\text{weight}_\lambda \det \mathcal{N}^\vee_S X |_{Z'} = \eta_\alpha$. Combining all these, we see that for every $j$, the weights of $\Lambda^j(\Omega_Y |_{Z'} \oplus \mathcal{N}^\vee_S X |_{Z'} \rightarrow \mathfrak{p}^\vee)$ are all $\leq \eta_\alpha$.

If $G$ is abelian, then $G=P$, $\mathfrak{n}^{-}=0$ and the weights of $\mathfrak{g}$ are all $0$. Then it suffices to know that the exterior powers of $\Omega_Y|_{Z'} \oplus (\mathcal{N}^\vee_{S}X)|_{Z'}$ have weights $\leq \eta_\alpha$ for the reasons above.
\end{proof}

\begin{corollary} \label{can use teleman}
Let $\mathcal{L} \in D^b(\mathfrak{X})$ be the linearization. Then the complex $\Lambda^j L_{\mathfrak{X}}\otimes \mathcal{L}$ satisfies the hypothesis of Theorem \ref{teleman}, so 
\begin{align*}
H^i(\mathfrak{X}^{ss},\Lambda^j L_{\mathfrak{X}^{ss}}\otimes \mathcal{L})=H^i(\mathfrak{X},\Lambda^j L_{\mathfrak{X}}\otimes \mathcal{L}). 
\end{align*}
If $G$ is abelian, we also have $H^i(\mathfrak{X}^{ss},(\Lambda^j L_{\mathfrak{X}^{ss}})^p\otimes \mathcal{L})=H^i(\mathfrak{X},(\Lambda^j L_{\mathfrak{X}})^p\otimes \mathcal{L})$ for each $p$.
\end{corollary}

\begin{proof}
Indeed, by definition of the stratification, $\text{weight}_{\lambda_\alpha} \mathcal{L} |_{Z_\alpha} <0$ for every $\alpha$, and weights are additive with respect to tensor product.
\end{proof}

\section{The Koszul complex of sections}\label{koszul section}

Let $\mathcal{L}$ be a $G$-linearized ample line bundle on a smooth projective-over-affine variety $X$ and consider the complex $\Lambda^j L_\mathfrak{X} \otimes \mathcal{L}$. We want to investigate the associated complex $F^\cdot$ of global sections,
\begin{align} \label{koszul F}
F^\cdot=\left[ 0\rightarrow H^0(X,\Omega^j_X \otimes \mathcal{L}) \rightarrow H^0(X,\Omega_X^{j-1} \otimes \mathcal{L})\otimes \mathfrak{g}^\vee \rightarrow \cdots \rightarrow H^0(X,\mathcal{L})\otimes S^j\mathfrak{g}^\vee\rightarrow 0 \right],
\end{align}
concentrated in degrees $0$ to $j$. For the remainder of the section, we extend the definition of Bott vanishing to a smooth projective-over-affine variety using equation (\ref{vanishing}), that is, $X$ satisfies Bott vanishing if (\ref{vanishing}) is valid for every ample line bundle on $X$.

\begin{lemma} \label{F computes hypercohomology}
Suppose $X$ satisfies Bott vanishing. Then the hypercohomology of $\Lambda^j L_\mathfrak{X}\otimes \mathcal{L}$ equals the $G$-equivariant cohomology of $F^\cdot$, that is, $H^i(\mathfrak{X},\Lambda^j L_\mathfrak{X}\otimes \mathcal{L})= \mathcal{H}^i(F^\cdot)^G$.
\end{lemma}

\begin{proof}
Consider $\Lambda^j L_\mathfrak{X} \otimes \mathcal{L}$ as a complex of coherent sheaves on $X$. From a suitable bi-complex resolution we get a spectral sequence $E_1^{p,q}=H^q(X,(\Lambda^j L_\mathfrak{X})^p\otimes \mathcal{L})$ converging to the hypercohomology $H^{p+q}(X,\Lambda^j L_\mathfrak{X}\otimes \mathcal{L})$. Since $X$ itself satisfies Bott vanishing, all the terms $E_1^{p,q}$ are equal to zero except for $q=0$:
\[
\begin{tikzcd}[cramped,sep=scriptsize]
0 \arrow[r] & 0         \arrow[r] & 0\arrow[r] & \cdots \arrow[r] & 0 \arrow[r] & 0\phantom{.}\\
0 \arrow[r] & E_1^{0,0} \arrow[r] & E_1^{1,0}\arrow[r] & \cdots \arrow[r] & E_1^{j,0} \arrow[r] & 0.
\end{tikzcd}
\]

Therefore the sequence degenerates at $E_2$. We get $E_\infty^{p,0}=E_2^{p,0}= \mathcal{H}^p(E_1^{\cdot,0})$. The hypercohomology $H^p(\mathfrak{X},\Lambda^j L_\mathfrak{X}\otimes \mathcal{L})$ equals the invariant sections $H^p({X},\Lambda^j L_\mathfrak{X}\otimes \mathcal{L})^G = \mathcal{H}^p(E_1^{\cdot,0})^G$. But the complex $E_1^{\cdot,0}$ is precisely $F^\cdot$.
\end{proof}

Now write $\mathbb{P}^m= \mathbb{P}(\mathfrak{g})$ (the space of lines in $\mathfrak{g}$) and let $W = X \times \mathbb{P}(\mathfrak{g})$, carrying a canonical $G$-action. Observe $H^0(\mathbb{P}(\mathfrak{g}),\mathcal{O}_{\mathbb{P}(\mathfrak{g})}(l))=S^l\mathfrak{g}^\vee$ for each $l\geq 0$. Given the action, the vector bundle $T_X \boxtimes \mathcal{O}_{\mathbb{P}(\mathfrak{g})}(1)$ has a canonical $G$-equivariant global section $s \in H^0(X\times\mathbb{P}(\mathfrak{g}),T_X\boxtimes \mathcal{O}_{\mathbb{P}(\mathfrak{g})}(1))^G=(H^0(X,T_X)\otimes \mathfrak{g}^\vee)^G$, which is the one giving the map $\Omega_X \rightarrow \mathfrak{g}^\vee$. Let $M \subset W$ be the scheme-theoretic zero locus of $s$. Suppose this is a local complete intersection, that is, the section $s$ is given locally by a regular sequence. By smoothness of $X\times(\mathbb{P}^1)^n$, this is equivalent to $\codim M=n$, where $n=\rk (T_X\boxtimes\mathcal{O}_{\mathbb{P}(\mathfrak{g})}(1)) =\dim X$ (see e.g. \cite[\S II.8]{hartshorne}). In this case, the associated augmented Koszul complex
\begin{align} \label{K_s}
K_s^\cdot = \left[ 0\rightarrow\Omega_X^n \boxtimes \mathcal{O}_{\mathbb{P}(g)}(-n) \rightarrow \cdots \rightarrow \Omega_X \boxtimes \mathcal{O}_{\mathbb{P}(\mathfrak{g})}(-1) \rightarrow \mathcal{O}_W \rightarrow \mathcal{O}_M\rightarrow 0 \right]
\end{align}
is exact (see e.g. \cite[Corollary 4.5.4]{weibel}). We consider this complex to be concentrated in degrees $-n$ to $1$, that is, $K_s^p=\Omega_X^{-p}\otimes\mathcal{O}_{\mathbb{P}^m}(p)$ for $p\leq 0$ and $K_s^1=\mathcal{O}_M$.

\begin{proposition} \label{if lci}
Suppose $M$ is a local complete intersection and suppose $X$ satisfies Bott vanishing. Then $H^i(\mathfrak{X}^{ss},\Lambda^j L_{\mathfrak{X}^{ss}}\otimes \mathcal{L})=0$ for $i \neq 0, j$. If, in addition, $H^0(M,\mathcal{L}\boxtimes \mathcal{O}_{\mathbb{P}(\mathfrak{g})}(j)|_M)^G=0$, then $H^j(\mathfrak{X}^{ss},\Lambda^j L_{\mathfrak{X}^{ss}}\otimes \mathcal{L})=0$ too.
\end{proposition}

\begin{proof}
From Corollary \ref{can use teleman}, we know $H^i(\mathfrak{X}^{ss},\Lambda^j L_{\mathfrak{X}^{ss}}\otimes \mathcal{L})=H^i(\mathfrak{X},\Lambda^j L_{\mathfrak{X}}\otimes \mathcal{L})$. Therefore,
by Lemma \ref{F computes hypercohomology}, it suffices to show $\mathcal{H}^i(F^\cdot)^G=0$ for $0<i<j$. Since the Koszul complex $K_s^\cdot$ is exact, all its hypercohomologies vanish. The same is true for the complex $\tilde{K}_s^\cdot=K_s^\cdot \otimes (\mathcal{L} \boxtimes \mathcal{O}_{\mathbb{P}(\mathfrak{g})}(j))$. Take the associated spectral sequence $E_1^{p,q}=H^q(X\times \mathbb{P}(\mathfrak{g}),\tilde{K}_s^p)$, converging to $H^{p+q}(X\times \mathbb{P}(\mathfrak{g}),\tilde{K}_s^\cdot)=0$. Since $X$ satisfies Bott vanishing, we have
\begin{align*}
H^q(X\times \mathbb{P}(\mathfrak{g}),\tilde{K}_s^p)=
\begin{cases}
H^0(X,\Omega_X^{-p} \otimes \mathcal{L}) \otimes S^{p+j}\mathfrak{g}^\vee & \text{if} \quad q=0,\: -j\leq p\leq 0 \\
H^0(X,\Omega_X^{-p} \otimes \mathcal{L}) \otimes H^m(\mathbb{P}^m,\mathcal{O}_{\mathbb{P}^m}(j+p))^\vee & \text{if} \quad q=m,\: p\leq -j-m-1 \\
H^q(M,(\mathcal{L} \boxtimes \mathcal{O}_{\mathbb{P}^m}(j))|_M) & \text{if} \quad p=1 \\
0 & \text{otherwise,}
\end{cases}
\end{align*}
and the sequence has the following shape:
\[
\begin{tikzcd}[cramped,sep=small]
\cdots\arrow[r] &E_1^{-j-m-2,m}\arrow[r] &E_1^{-j-m-1,m}\arrow[r]&0\arrow[r] &\cdots & & &\cdots \arrow[r] &0 \arrow[r]  &E_1^{1,m}\\
 & & \cdots\arrow[r] &0\arrow[r] &\cdots & & &\cdots \arrow[r] &0 \arrow[r] &E_1^{1,m-1}\\
 & &                 &\vdots     &       & & &                 & \vdots     &\vdots \\
 & & \cdots\arrow[r] &0\arrow[r] &\cdots & & &\cdots \arrow[r] &0 \arrow[r] &E_1^{1,1}\\
 & & \cdots\arrow[r] &0\arrow[r] &\cdots\arrow[r] & 0 \arrow[r] &E_1^{-j,0} \arrow[r]  &\cdots \arrow[r] &E_1^{0,0}\arrow[r] &E_1^{1,0}.
\end{tikzcd}
\]

Note that the complex $F^\cdot$ is precisely the naive truncation of the shifted complex $E_1^{\cdot,0}[-j]$ obtained by omitting the last term, since the differentials are determined precisely by the section $s:\Omega_X \rightarrow \mathfrak{g}^\vee$.

From the description of the spectral sequence, we see that for $q=0$ and $-j+1 \leq p \leq 0$, it degenerates at $E_2$ and we get $0=H^{i-j}(X\times \mathbb{P}(\mathfrak{g}),\tilde{K}_s^\cdot)=\mathcal{H}^i(F^\cdot)$, for $1\leq i < j$ (even before taking invariants). By the same reason we find that indeed $H^q(M,(\mathcal{L} \boxtimes \mathcal{O}_{\mathbb{P}^m}(j))|_M)=0$ if $q>0$, although we do not need this. Now if, further, $\mathcal{L} \boxtimes \mathcal{O}_{\mathbb{P}(\mathfrak{g})}(j)|_M$ has no invariant global sections, then the complex of $G$-invariants $(E_1^{\cdot,0}[-j])^G$ is precisely $(F^\cdot)^G$, so in this case, $\mathcal{H}^j((F^\cdot)^G)=0$. Since taking invariant sections is an exact functor, this is the same as saying that $\mathcal{H}^j(F^\cdot)^G=0$.
\end{proof}

\subsection{Vanishing on the GIT quotient}

Observe that Proposition \ref{if lci} applies to the quotient stack $\mathfrak{X}^{ss}$. Now, if the action of $G$ on $X^{ss}$ is free, this result can be interpreted in terms of its coarse moduli space, namely the GIT quotient $Y=X\git_\mathcal{L}G$. Indeed, if $G$ acts freely on $X^{ss}$, then $\Lambda^j L_{\mathfrak{X}^{ss}}$ descends to $\Omega^j_Y$, as observed in \S \ref{git section}. Freeness of the action also implies that $\mathcal{L}$ descends to a line bundle $L$ on $Y$. By exactness of $p_*=\pi_*(\cdot)^G$, we have $H^i(Y,\Omega^j_Y\otimes L)=H^i(\mathfrak{X}^{ss},\Lambda^j L_{\mathfrak{X}^{ss}}\otimes \mathcal{L})$, so if the latter vanishes, then so does $H^i(Y,\Omega^j_Y\otimes L)$.

In the case of a quotient $Y=(\mathbb{P}^1)^n\git_\mathcal{L}PGL_2$ without strictly semi-stable locus, we can check that the hypotheses of Proposition \ref{if lci} are satisfied, so $H^i(Y,\Omega^j_Y\otimes L)=0$ for $i\neq 0,j$, where $L$ is the descent of the linearization $\mathcal{L}$. In Section \ref{main section}, we will see that, in order to show Bott vanishing for $(\mathbb{P}^1)^n\git_\mathcal{L}PGL_2$, the only line bundle we need to consider is precisely the descent of $\mathcal{L}$. The rest of that section will be devoted to prove the vanishing of the $j$-th cohomology.

More generally, when the action is not free on $X^{ss}$ and the coarse moduli space $Y$ is not smooth, we introduce the following notation.

\begin{notation}\label{notation coarse omega}
Let $Y^0 \subset Y$ be the nonsingular locus, with inclusion $\imath:Y^0 \hookrightarrow Y$. Then for each $j$, $\Omega^j_Y$ will denote the (non-derived) pushforward $\imath_*(\Omega^j_{Y^0})$,  known as the sheaf of reflexive differentials. We call $X'\subset X^{ss}$ the locus where $G$ acts freely.
\end{notation}

Note that $X'\subset\pi^{-1}(Y^0)$, where $\pi:X^{ss}\rightarrow Y$ is the quotient map. Recall $\Lambda^j L_{\mathfrak{X}^{ss}}$ is a vector bundle provided $X^{ss}=X^s$. We are interested in the cases when $\pi_*(\Lambda^j L_{\mathfrak{X}^{ss}})^G=\Omega^j_Y$. Suppose this holds and $\mathcal{L}$ descends to $L$, that is, there is a $G$-equivariant isomorphism $\pi^*(L)\cong\mathcal{L}|_{\mathfrak{X}^{ss}}$. In this situation, the projection formula yields $\pi_*(\Lambda^j L_{\mathfrak{X}^{ss}}\otimes\mathcal{L})^G=\Omega^j_Y\otimes L$ and Proposition \ref{if lci} can be interpreted as vanishing of cohomologies $H^i(Y,\Omega_Y^j\otimes L)$.

\begin{proposition}\label{cohomology GIT}
With the notation as above, suppose $X^{ss}=X^s$ and $X^{ss}\backslash X'\subset X^{ss}$ has codimension at least $2$. If $\mathcal{L}$ descends to a line bundle $L$ on $Y$, then $H^i(Y,\Omega^j_Y\otimes L)=H^i(\mathfrak{X}^{ss},\Lambda^j L_{\mathfrak{X}^{ss}}\otimes\mathcal{L})$ for every $i,j$.
\end{proposition}

\begin{proof}
Let $Y'=\pi(X')$ and consider the open inclusions $X'\hookrightarrow X^{ss}$ and $\iota:Y'\hookrightarrow Y$,
\[
\begin{tikzcd}[cramped,sep=scriptsize]
X'\arrow[hook]{r}\arrow{d}{\pi} &X^{ss}\arrow{d}{\pi} \\
Y'\arrow[hook]{r}{\iota} &Y \\
\end{tikzcd}
\]
where $Y'\subset Y^0\subset Y$ and $X'\subset \pi^{-1}(Y^0)\subset X^{ss}$. We first observe that, since $X^{ss}=X^{s}$, $\pi$ is equidimensional and so $Y\backslash Y'\subset Y$ has codimension at least $2$. The same is true for $Y^0\backslash Y'\subset Y^0$. Write $\iota$ as a composition
\begin{align*}
Y'\overset{\imath'}\hookrightarrow Y^0\overset{\imath}\hookrightarrow Y.
\end{align*}
By smoothness of $X$, $Y$ has to be normal, and then we see that $\iota_*\mathcal{O}_{Y'}=\mathcal{O}_{Y}$, while $\imath'_*\Omega^j_{Y'}=\Omega^j_{Y^0}$, by the codimension condition. Using $\iota=\imath\circ\imath'$, we get $\iota_*\Omega^j_{Y'}=\Omega^j_Y$.

$G$ acts freely on $X'$, so we have a $G$-equivariant short exact sequence $0\rightarrow \pi^*\Omega_{Y'}\rightarrow \Omega_{X'}\rightarrow \mathfrak{g}^\vee\rightarrow 0$. Therefore, the restriction of $L_{\mathfrak{X}^{ss}}$ to $X'$ descends to $\Omega_{Y'}$, and similarly for their $j$-th exterior powers. We then have $\pi_*(\Lambda^j L_{\mathfrak{X}^{ss}}|_{X'})^G=\Omega^j_{Y'}$. On the other hand, it is not difficult to see that $\pi_*(\Lambda^j L_{\mathfrak{X}^{ss}}|_{X'})^G=\pi_*(\Lambda^j L_{\mathfrak{X}^{ss}})^G|_{Y'}$. That is,
\begin{align}\label{omega git}
\iota^*\pi_*(\Lambda^j L_{\mathfrak{X}^{ss}})^G=\Omega^j_{Y'}.
\end{align}
Using the projection formula on (\ref{omega git}) and the fact that $\iota_*\mathcal{O}_{Y'}=\mathcal{O}_{Y}$, we get $\pi_*(\Lambda^j L_{\mathfrak{X}^{ss}})^G=\iota_*\Omega^j_{Y'}=\Omega^j_Y$. Therefore,
\begin{align*}
\pi_*(\Lambda^j L_{\mathfrak{X}^{ss}}\otimes \mathcal{L})^G=\Omega^j_Y \otimes L
\end{align*}
again by the projection formula. By exactness of $p_*=\pi_*(\cdot)^G$, we obtain $H^i(Y,\Omega^j_Y\otimes L)=H^i(\mathfrak{X}^{ss},\Lambda^j L_{\mathfrak{X}^{ss}}\otimes \mathcal{L})$.
\end{proof}

\begin{remark}\label{remark X'}
If $X^{ss}=X^s$ and the action is free on $\pi^{-1}(Y^0)$, we have that $X'=\pi^{-1}(Y^0)$ and the condition on the codimension of $X^{ss}\backslash X'\subset X^{ss}$ is automatically satisfied. Indeed, $Y\backslash Y^0\subset Y$ has codimension $\geq 2$ by normality of $Y$. Equidimensionality of $\pi$ guarantees that the same is true for $X^{ss}\backslash\pi^{-1}(Y^0)\subset X^{ss}$.
\end{remark}

\section{The case of $X$ affine and $G$ abelian}\label{abelian section}

In the case that $G$ is an abelian group and $X$ is a smooth affine variety, we get a stronger version of Proposition \ref{if lci}, provided there is no strictly semi-stable locus. To do this, we apply very similar techniques to the ones used in Section \ref{koszul section}. The difference is that, in this case, we can take advantage of Corollary \ref{can use teleman} by working on the semi-stable locus from the beginning. Also, an affine variety $X$ automatically satisfies the Bott vanishing condition.

As usual, $\mathcal{L}$ denotes a $G$-linearized ample line bundle on a smooth projective-over-affine variety $X$ with a $G$-action. Consider the augmented Koszul complex $K_s^\cdot$ on $X\times \mathbb{P}(\mathfrak{g})$ defined in (\ref{K_s}), and let ${\bar{K}_s}^\cdot$ be its restriction to $X^{ss}\times \mathbb{P}(\mathfrak{g})$. We first observe that this restriction is exact if $X^{ss}=X^s$. This is because the projection of $M$ to $X$ lands entirely on the unstable locus.

\begin{lemma}\label{M empty abelian}
If $X^{ss}=X^s$, then $M\cap (X^{ss}\times \mathbb{P}(\mathfrak{g}))=\emptyset$. In particular, the restriction ${\bar{K}_s}^\cdot$ is acyclic in this case.
\end{lemma}

\begin{proof}
For a pair $(x,l)$ in $M$, $l$ must be a line in $\mathfrak{g}=\mathbb{A}^{m+1}$ contained in the Lie algebra of the stabilizer $G_x$, so $x$ cannot be stable. By the assumption, $x \notin X^{ss}$. As a consequence, the restriction of $K_s^\cdot$ to $X^{ss}\times \mathbb{P}(\mathfrak{g})$ is acyclic since $M \cap (X^{ss}\times \mathbb{P}(\mathfrak{g}))=\emptyset$ is a local complete intersection.
\end{proof}

Now suppose $G$ is abelian and $X$ is affine. Let $\bar{F}^\cdot$ the complex of global sections of $\Lambda^j L_{\mathfrak{X}^{ss}} \otimes \mathcal{L}$,
\begin{align} \label{koszul F abelian}
\bar{F}^\cdot=\left[ 0\rightarrow H^0(X^{ss},\Omega^j_{X^{ss}} \otimes \mathcal{L}) \rightarrow H^0(X^{ss},\Omega_{X^{ss}}^{j-1} \otimes \mathcal{L})\otimes \mathfrak{g}^\vee \rightarrow \cdots \rightarrow H^0(X^{ss},\mathcal{L})\otimes S^j\mathfrak{g}^\vee\rightarrow 0\right],
\end{align}
concentrated in degrees $0$ to $j$. Using a similar argument to the one in Lemma \ref{F computes hypercohomology}, we can show the complex of invariants $(\bar{F}^\cdot)^G$ computes the hypercohomologies of $\Lambda^j L_{\mathfrak{X}^{ss}}\otimes \mathcal{L}$.

\begin{lemma} \label{F computes hypercohomology abelian}
If $G$ is abelian and $X$ is affine, we have $H^i(\mathfrak{X}^{ss},\Lambda^j L_{\mathfrak{X}^{ss}}\otimes \mathcal{L})=\mathcal{H}^i(\bar{F}^\cdot)^G$.
\end{lemma}

\begin{proof}
First, we observe that by Corollary \ref{can use teleman},
using that $G$ is abelian, $H^i(\mathfrak{X}^{ss},(\Lambda^j L_{\mathfrak{X}^{ss}})^p\otimes \mathcal{L})=H^i(\mathfrak{X},(\Lambda^j L_{\mathfrak{X}})^p\otimes \mathcal{L})$. For $i>0$ this is zero since $X$ is affine. Now take the spectral sequence $E_1^{p,q}=H^q(\mathfrak{X}^{ss},(\Lambda^j L_{\mathfrak{X}^{ss}})^p\otimes \mathcal{L})$, which converges to $H^{p+q}(\mathfrak{X}^{ss},\Lambda^j L_{\mathfrak{X}^{ss}}\otimes \mathcal{L})$. By the previous observation, we see that $E_1^{p,q}=0$ for $q\neq 0$, and so $H^i(\mathfrak{X}^{ss},\Lambda^j L_{\mathfrak{X}^{ss}}\otimes \mathcal{L})=\mathcal{H}^i(E_1^{\cdot,0})$ for every $i$. But the complex $E_1^{\cdot,0}$ is precisely $(\bar{F}^\cdot)^G$.
\end{proof}

Following the ideas from Proposition \ref{if lci}, we obtain the following vanishing result. Here $\Omega^j_Y$ and $X'$ are as in Notation \ref{notation coarse omega}. 

\begin{theorem} \label{theorem bott for abelian}
Suppose $G$ is abelian, $X$ is affine and $X^{ss}=X^s$. Then $H^i(\mathfrak{X}^{ss},\Lambda^j L_{\mathfrak{X}^{ss}}\otimes \mathcal{L})=0$ for every $i>0,j\geq 0$. Further, if $X^{ss}\backslash X'$ has codimension at least $2$ and $\mathcal{L}$ descends to $L$, we have
\begin{align*}
H^i(Y,\Omega^j_Y\otimes L)=0\quad\forall i>0,j\geq 0.
\end{align*}
\end{theorem}

\begin{proof}
Let ${\bar{K}_s}^\cdot$ be the restriction of $K_s^\cdot$ to $X^{ss}\times \mathbb{P}(\mathfrak{g})$. By Lemma \ref{M empty abelian}, ${\bar{K}_s}^\cdot$ is an acyclic complex, being the (augmented) Koszul resolution of $M \cap (X\times \mathbb{P}(\mathfrak{g}))=\emptyset$. 

Take now the spectral sequence $E_1^{p,q}=H^q(X^{ss}\times \mathbb{P}(\mathfrak{g}),{\bar{K}_s}^p\otimes \mathcal{L})^G$, converging to $H^{p+q}(X^{ss}\times \mathbb{P}(\mathfrak{g}),{\bar{K}_s}^\cdot\otimes\mathcal{L})^G=0$. Since $H^i({X}^{ss},(\Lambda^j L_{\mathfrak{X}^{ss}})^p\otimes \mathcal{L})^G=0$ for $i>0$, we find
\begin{align*}
E_1^{p,q}=
\begin{cases}
(H^0(X^{ss},\Omega_{X^{ss}}^{-p} \otimes \mathcal{L}) \otimes S^{p+j}\mathfrak{g}^\vee)^G & \text{if} \quad q=0,\: -j\leq p\leq 0 \\
(H^0(X^{ss},\Omega_{X^{ss}}^{-p} \otimes \mathcal{L}) \otimes H^m(\mathbb{P}^m,\mathcal{O}_{\mathbb{P}^m}(j+p))^\vee)^G & \text{if} \quad q=m,\: p\leq -j-m-1 \\
0 & \text{otherwise.}
\end{cases}
\end{align*}
Note that the complex of invariants $(\bar{F}^\cdot)^G$ is precisely the shifted complex $E_1^{\cdot,0}[-j]$. For $q=0$ and $-j+1 \leq p \leq 0$, the sequence degenerates at $E_2$ and we get $0=H^{i-j}(X^{ss}\times \mathbb{P}(\mathfrak{g}),{\bar{K}_s}^\cdot\otimes\mathcal{L})=\mathcal{H}^i(\bar{F}^\cdot)^G$, for $i\geq 1$. From Lemma \ref{F computes hypercohomology abelian}, we conclude $H^i(\mathfrak{X}^{ss},\Lambda^j L_{\mathfrak{X}^{ss}}\otimes \mathcal{L})=0$ for $i>0$. The last part of the statement is a direct consequence of Proposition \ref{cohomology GIT}.
\end{proof}

\subsection{The toric case}\label{toric subsection}

Now let $Y$ be a $\mathbb{Q}$-factorial projective toric variety. From \cite{cox-coxring}, we know $Y$ is the GIT quotient of an affine space $X=\mathbb{A}^d$ by the abelian reductive group $G=\Hom (\Cl Y, \mathbb{G}_m)$, with $X^s=X^{ss}$ and with $X^{us}\subset X$ of codimension at least $2$. The character group of $G$ is canonically identified with $\Cl Y$. If we call $\Sigma$ the fan in $N \cong \mathbb{Z}^n$ determining the toric variety and $d=|\Sigma (1)|$ the number of $1$-dimensional cones, then we have a surjection $\mathbb{Z}^d=\langle e_\rho, \: \rho \in \Sigma(1) \rangle_\mathbb{Z} \twoheadrightarrow \Cl Y$ and we can write $X=\spec R$, where $R=\Bbbk[x_1,\ldots,x_d]=\bigoplus_{v \in \Cl Y} R_v$ is the Cox ring, and each graded piece $R_v \cong H^0(Y,\mathcal{O}_Y(D))$, for $v=[D]$. We have a short exact sequence
\begin{align}\label{torus ses}
0\rightarrow M \rightarrow \mathbb{Z}^d\rightarrow \Cl Y \rightarrow 0,
\end{align}
with the map on the left being $m\mapsto \sum \langle m,n_\rho \rangle e_\rho$, where $n_\rho \in N$ is the vector corresponding to the $1$-dimensional cone $\rho \in \Sigma (1)$. This way, the action of $G$ is described by the short exact sequence
\begin{align*}
1\rightarrow G \rightarrow \left(\mathbb{G}_m\right)^d\rightarrow T \rightarrow 1
\end{align*}
obtained by applying $\Hom(\cdot,\mathbb{G}_m)$ to (\ref{torus ses}). Here $T=N\otimes \mathbb{G}_m$ is the torus acting on $Y$. Using the usual description of $Y$ as $\bigcup u_\sigma$, where $u_\sigma=\spec \Bbbk[\sigma^\vee\cap M]$, this quotient is described locally by $u_\sigma=U_\sigma /G$, where $U_\sigma=\{z \in \mathbb{A}^d \mid z_\rho \neq 0\quad\forall\rho \notin \sigma(1)\}$ (see \cite{cox-coxring}, \cite[\S 14]{cox-book} for details). In the case that $Y$ is smooth, then $G=\mathbb{G}_m^{d-n}$ is a torus and the action is free on $X^{ss}$.

In what follows we will see that, from Theorem \ref{theorem bott for abelian} we can recover Bott vanishing for the toric case, over $\Bbbk$. For the remainder of the present section $X$ will denote $\mathbb{A}^d$, $G$ will denote $\Hom (\Cl Y, \mathbb{G}_m)$ and $Y=X\git_\mathcal{L} G$ will be a projective $\mathbb{Q}$-factorial toric variety obtained as a GIT quotient given by a linearization $\mathcal{L}$.

\begin{definition}\label{ample weil}
We say that a Weil divisor $D$ on $Y$ is ample if a positive multiple $mD$ is an ample Cartier divisor.
\end{definition}

We first see that, in order to show Bott vanishing, the only ample divisor we need to consider is the descent of $\mathcal{L}$ (cf. \cite[Proposition 2.9]{hu-keel}). 

\begin{lemma}\label{L=L' toric}
With the notation as above, let $D$ be an ample divisor on $Y$. Then $\mathcal{O}_Y(D)$ is the descent of a linearization $\mathcal{L}'$ such that $Y=X\git_{\mathcal{L}'}G$.
\end{lemma}

\begin{proof}
Since $X$ is an affine space and $X^{us}$ has codimension $\geq 2$, we see $\Pic X = \Pic X^{ss}$ is trivial and the $G$-equivariant Picard group is $\Pic^G X=\Pic^G X^{ss}=\Cl Y$, the character group of $G$. The map $\Pic Y\rightarrow \Pic^G X^{ss}$, $L\mapsto \pi^*L$, is the inclusion $\Pic Y \hookrightarrow \Cl Y$. That is, for any Weil divisor $D$, the sheaf $\mathcal{O}_Y(D)$ is the descent of $\mathcal{L}_v$, which is the trivial line bundle on $X$ linearized by the character $v=[D] \in \Cl Y$. Further, given a linearization $\mathcal{L}_w$, for some $w \in \Cl Y$, we see that $R^G$ is precisely $\bigoplus_{k \geq 0} R_{kw}=\bigoplus_{k\geq 0}H^0(Y,\mathcal{O}_Y{(kD)})$, for $w=[D]$. Now if $L=\mathcal{O}_Y(mD)$, $m>0$, is an ample line bundle on $Y$, then $Y=\proj \bigoplus_{k\geq 0}H^0(Y,L^{\otimes k})\cong \proj \bigoplus_{k\geq 0}H^0(Y,\mathcal{O}_Y{(kD)})$, so that $\mathcal{O}_Y(D)$ is the descent of a linearization $\mathcal{L}' \in \Pic^G X$ such that $Y=X\git_{\mathcal{L}'}G$.
\end{proof}

We also check that the action is free on the preimage of the smooth locus $Y^0\subset Y$.

\begin{lemma}\label{free action toric}
$G$ acts freely on $\pi^{-1}(Y^0)$.
\end{lemma}

\begin{proof}
The smooth locus of $Y$ is given by $\bigcup_{\sigma \text{ smooth}} u_\sigma$ (see e.g. \cite[Proposition 11.1.2]{cox-book}). It suffices to check that $G$ acts freely on $U_\sigma$ for $\sigma \in \Sigma$ a smooth cone. Consider the map $h:\mathbb{Z}^d\twoheadrightarrow \Cl Y$ from (\ref{torus ses}) and suppose $g \in G=\Hom(\Cl Y,\mathbb{G}_m)$ is in the stabilizer of some $z\in U_\sigma$. Since $z_\rho\neq 0$ for every $\rho \notin \sigma(1)$, this implies $g(v)=1$ for every $v \in h(\langle e_\rho, \: \rho\notin\sigma(1)\rangle_\mathbb{Z})$. 

But in fact, if $\sigma$ is a smooth cone, the restriction of $h$ to the span of $\{e_\rho,\:\rho\notin\sigma(1)\}$ is still surjective. This is because we can complete $\{n_\rho,\:\rho\in\sigma(1)\}$ to a $\mathbb{Z}$-basis $\{n_\rho\}\cup\{n'_\alpha\}$ of $N$, choose a dual basis $\{m_\rho,\:\rho\in\sigma(1)\}\cup\{m'_\alpha\}$ and see that under the map $f:M\rightarrow\mathbb{Z}^d$ from (\ref{torus ses}), $m_\rho\mapsto e_\rho + w$, some $w \in \langle e_\rho, \: \rho\notin\sigma(1)\rangle_\mathbb{Z}$. As a consequence, every vector $w\in\mathbb{Z}^d$ can be written as $w'+w''$, with $w'\in \im (f)=\ker(h)$, $w''\in \langle e_\rho, \: \rho\notin\sigma(1)\rangle_\mathbb{Z}$. We conclude $\langle e_\rho, \: \rho\notin\sigma(1)\rangle_\mathbb{Z} \rightarrow \Cl Y$ is surjective. Therefore $g(v)=1$ for every $v\in \Cl Y$, so $g=1$.
\end{proof}

Finally, we get a new proof of the following well-known result.

\begin{theorem}[(Bott vanishing for toric varieties)]\label{thm bott for toric}
Let $Y$ be a $\mathbb{Q}$-factorial projective toric variety over $\Bbbk$ and $L$ an ample line bundle on $Y$. Then $H^i(Y,\Omega_Y^j\otimes L)=0$ for every $i>0,j\geq 0$. In particular, a smooth projective toric variety over $\Bbbk$ satisfies Bott vanishing.
\end{theorem}

\begin{proof}
Let $L$ be an ample line bundle on $Y$. By the discussion above $L$ is the descent of a linearization $\mathcal{L}'$ such that $Y=X\git_{\mathcal{L}'}G$, so we can assume $L$ is the descent of the linearization $\mathcal{L}$. By Lemma \ref{free action toric}, the non-free locus $X^{ss}\backslash X'$ has codimension $\geq 2$ (see Remark \ref{remark X'}), so Theorem \ref{theorem bott for abelian} implies $H^i(Y,\Omega_Y^j\otimes L)=0$ for $i>0,j\geq 0$. If $Y$ is smooth, then $Y^0=Y$ and this is Bott vanishing.
\end{proof}

\begin{remark}
In the statement of Theorem \ref{thm bott for toric} we may replace $L$ by the reflexive sheaf $\mathcal{O}_Y(D)$, where $D$ is any ample Weil divisor, understood as in Definition \ref{ample weil}. Indeed, the same proof holds since Lemma \ref{L=L' toric} is valid for this case as well.
\end{remark}

\section{The case of $X= (\mathbb{P}^1)^n$ and $G=PGL_2$}\label{main section}

Now we consider the diagonal action of $PGL_2$ on $(\mathbb{P}^1)^n$, so throughout this section $G$ will denote $PGL_2$, $X$ will denote $(\mathbb{P}^1)^n$ and $\mathfrak{g}$ will denote $\mathfrak{sl}_2$. For a given ample line bundle $\mathcal{L}=\mathcal{O}_X(d_1,\ldots,d_n)$, $d_i>0$, where $\sum d_i$ is even, there is a unique $PGL_2$-linearization, giving rise to a GIT quotient $Y=X \git_\mathcal{L} PGL_2$, a projective variety. Variation of GIT is described, for instance, in \cite[\S 8]{hassett}

A maximal torus of $\mathfrak{g}$ is one-dimensional, so to get a KN stratification it essentially suffices to consider a single one-parameter subgroup. We consider $\lambda: \mathbb{G}_m \rightarrow PGL_2$ given by
\begin{align*}
\lambda(t)=\begin{bmatrix}
t & 0 \\
0 & t^{-1}
\end{bmatrix}.
\end{align*}
The fixed locus of $\lambda$ is the union of the points $z_I$ where, for every $I \subset \{1,\ldots,n\}$, $z_I$ has coordinates $z_i=\infty$ if $i \in I$ and $z_i = 0$ otherwise. We use the convention $0=(0:1)$, $\infty=(1:0)$. One can compute
\begin{align*}
\mu(\lambda,I)=-\text{weight}_\lambda \mathcal{L}|_{z_I} = \sum_{i \notin I} d_i - \sum_{i \in I} d_i
\end{align*}
and we can get a KN stratification of the unstable locus indexed by the subsets $I$ for which $\mu(\lambda,I) >0$. Indeed, a point $z=(z_1,\ldots,z_n)\in X$ is unstable if and only if there is an $I\subset \{1,\ldots,n\}$ with $\sum_{i\in I}d_i>\sum_{i\notin I}d_i$ such that $z_i=\alpha$ for every $i\in I$. Also, it can be computed that $\eta_{\lambda,I}=2(|I|-1)$ (see \cite{castravet2} for details). Since the ambient space $X=(\mathbb{P}^1)^n$ is a smooth projective toric variety, it satisfies Bott vanishing and then results from Section \ref{koszul section} can be applied.

Note that in our case, the cotangent sheaf is a direct sum of line bundles, namely $\Omega_X = \bigoplus_{i=1}^n \mathcal{O}_X(0,\ldots,-2,\ldots,0)$, each summand having a $-2$ in the $i$-th position and zeros elsewhere. The section $s\in H^0(T_X\boxtimes \mathcal{O}_{\mathbb{P}(\mathfrak{g})}(1))^G$ associated to the map $\Omega_X \rightarrow \mathfrak{g}^\vee$ is then the sum of $n$ sections $s_i$, where each $s_i\in (H^0(X,\mathcal{O}_X(0,\ldots,2,\ldots,0))\otimes \mathfrak{g}^\vee)^G$.

Let $\{E,H,F\}$ be the usual basis of $\mathfrak{g}$, where $[E,H]=-2E$, $[E,F]=H$, $[H,F]=-2F$. If we choose a basis $\{X_0,Y_0,Z_0\}$ of $\mathfrak{g}^\vee=H^0(\mathbb{P}(\mathfrak{g}),T_{\mathbb{P}(\mathfrak{g})})$ that is dual to the basis $\{-E,H,F\}$, then we can explicitly compute $s_i$ in coordinates. For this, we use the isomorphism $T_{\mathbb{P}^1}\cong\mathcal{O}_{\mathbb{P}^1}(2)$, ${\partial}/{\partial (x/y)}\mapsto -y^2$, where $(x:y)$ are coordinates in $\mathbb{P}^1$. Writing
\begin{align*}
E=\frac{\partial}{\partial t}
\left.\begin{pmatrix}
1 & t\\
0 & 1
\end{pmatrix}\right|_{t=0},
\quad
H=\frac{\partial}{\partial t}
\left.\frac{1}{1-t^2}\begin{pmatrix}
1+t & 0\\
0 & 1-t
\end{pmatrix}\right|_{t=0},
\quad
F=\frac{\partial}{\partial t}
\left.\begin{pmatrix}
1 & 0\\
t & 1
\end{pmatrix}\right|_{t=0}
\end{align*}
and using the chart $y\neq 0$ we find that the action of $PGL_2$ on $\mathbb{P}^1$ determines the map $\mathfrak{g}\rightarrow T_{\mathbb{P}^1}$ that sends 
\begin{align*}
E\mapsto\frac{\partial}{\partial(x/y)},\quad H\mapsto\frac{2x}{y}\frac{\partial}{\partial(x/y)},\quad F\mapsto -\frac{x^2}{y^2}\frac{\partial}{\partial(x/y)}.
\end{align*}
Combining all these and using $\partial/\partial(x_i/y_i)\mapsto -y_i^2$ on each $i$-th component $\mathbb{P}^1_{(x_i:y_i)}$, we find $s_i = x_i^2Z_0 - 2x_iy_iY_0+y_i^2X_0$. Observe that we can also identify $\mathfrak{g}\cong \mathfrak{g}^\vee$ as $\mathfrak{g}$-representations by sending the basis $\{E,H,F\}$ to $\{Z_0,2Y_0,-X_0\}$.

\begin{remark}
Consider the diagonal action of $SL_2$ on $X$. Then any ample line bundle $\mathcal{L}=\mathcal{O}_X(d_1,\ldots,d_n)$ carries a unique $SL_2$-linearization, giving rise to a GIT quotient $Y=X\git_\mathcal{L} SL_2$. If $\sum d_i$ is even, then $\mathcal{L}$ also admits a unique $PGL_2$-linearization, and $X \git_\mathcal{L} SL_2 = X \git_\mathcal{L} PGL_2$. In any case, $\mathcal{L}^{\otimes 2}$ admits a $PGL_2$-linearization and $X \git_\mathcal{L} SL_2$ is canonically isomorphic to the quotient $X \git_{\mathcal{L}^{\otimes 2}} SL_2 = X \git_{\mathcal{L}^{\otimes 2}} PGL_2$.
\end{remark}

\begin{proposition} \label{M is lci}
Consider $X=(\mathbb{P}^1)^n$, $G=PGL_2$, $\mathfrak{g}=\mathfrak{sl}_2$ as above. Let $M \subset X \times \mathbb{P}(\mathfrak{g})$ be the vanishing locus of the section $s \in H^0(T_X \boxtimes \mathcal{O}_{\mathbb{P}(\mathfrak{g})}(1))^G$ associated to the map $\Omega_X \rightarrow \mathfrak{g}^\vee$. Then $M=\bigcap (s_i=0)$ is a local complete intersection.
\end{proposition}

\begin{proof}
The section $s$ is the direct sum of the $n$ sections $s_i \in H^0(\mathcal{O}_X(0,\dots,2,\ldots,0)\boxtimes \mathcal{O}_{\mathbb{P}^2}(1))^{G}$ given by $s_i = x_i^2Z_0 - 2x_iy_iY_0+y_i^2X_0$, as noted above. By smoothness of $X$, it suffices to check that $\dim M=2$. Consider the map $p:M \rightarrow \mathbb{P}^2$ given by the projection on the second component. We show $p$ is a finite map. Indeed, since $p$ is projective, it suffices to show that it has finite fibers. We note that a given point $(x_1:y_1;\ldots;x_n:y_n)\times (X_0:Y_0:Z_0)$ is in $M$ if and only if for every $i$, $(X_0:Y_0:Z_0) \in \mathbb{P}^2$ is in the line that is tangent to the rational normal curve $(X_0Z_0-Y_0^2=0)\subset \mathbb{P}^2$ at the point $(x_i^2:x_iy_i:y_i^2)$. Since every point is in at most $2$ lines tangent to a given conic, we have $|p^{-1}(X_0:Y_0:Z_0)| \leq 2^n$. Therefore $p$ is finite and $\dim M = 2$.
\end{proof}

\begin{corollary} \label{almost bott}
For a $PGL_2$-linearized ample line bundle $\mathcal{L}$ on $X$, we have $H^i(\mathfrak{X},\Lambda^j L_{\mathfrak{X}} \otimes \mathcal{L})=0$ for $i\neq 0,j$.
\end{corollary}

\begin{proof}
This follows from Proposition \ref{if lci}, as $M$ is a local complete intersection and $X=(\mathbb{P}^1)^n$ has the Bott vanishing property. Recall $H^i(\mathfrak{X}^{ss},\Lambda^j L_{\mathfrak{X}^{ss}}\otimes \mathcal{L})=H^i(\mathfrak{X},\Lambda^j L_{\mathfrak{X}}\otimes \mathcal{L})$ by Corollary \ref{can use teleman}.
\end{proof}

\subsection{The ring of invariants}

Let $Y=(\mathbb{P}^1)^n\git PGL_2$ be given by a polarization $\mathcal{L}=\mathcal{O}(d_1,\ldots,d_n)$, $d_i>0$. We will assume $X^{ss}=X^{s}$. This implies that the action of $G=PGL_2$ is free in $X^{ss}$ and $Y$ is smooth.

\begin{remark} \label{no sss}
The condition $X^{ss}=X^s$ is equivalent to the following condition: there is no partition $I \sqcup I^c = \{1,\ldots,n\}$ such that $\sum_{i \in I} d_i = \sum_{i \notin I} d_i$. This is a consequence of the Hilbert-Mumford criterion and the description of the unstable locus (see e.g. \cite[\S 4]{castravet2}).
\end{remark}

If we consider the action of the torus $(\mathbb{G}_m)^n$ on the Grassmannian $\Gr(2,n)$ and linearize the ample line bundle $\mathcal{O}_{\Gr(2,n)}(1)$ of the Pl\"ucker embedding using some character $(l_1,\ldots,l_n)$, we have Gelfand-MacPherson correspondence \cite[Theorem 2.4.7]{kapranov}:
\begin{align*}
\bigoplus_{d \geq 0} H^0((\mathbb{P}^1)^n,\mathcal{O}(dl_1,\ldots,dl_n))^{PGL_2} = \bigoplus_{d \geq 0} H^0(\Gr(2,n),\mathcal{O}(d))^{(\mathbb{G}_m)^n}.
\end{align*}

That is, $\bigoplus_{d \geq 0} H^0((\mathbb{P}^1)^s,\mathcal{O}(dl_1,\ldots,dl_n))^{PGL_2}$ can be seen as a subring of the homogeneous coordinate ring of the Grassmannian, $\Bbbk[p_{ik}]/(p_{ik}p_{rl}-p_{ir}p_{kl}+p_{il}p_{kr})$, where $p_{ik}=x_iy_k-x_ky_i$ are the Pl\"ucker minors. The $d$-th graded piece corresponds to polynomials in $p_{ik}$ having multi-degree $dl_1,\ldots,dl_n$ in $x_1,y_1;\ldots;x_n,y_n$.

\begin{lemma} \label{is toric}
Suppose we have a linearization $\mathcal{L}$ giving $X^{ss}=X^s$ and with an unstable locus having an irreducible component of codimension $1$. Then $Y=(\mathbb{P}^1)^n \git_\mathcal{L} PGL_2$ is a smooth projective toric variety.
\end{lemma}

\begin{proof}
Given that $X^s=X^{ss}$, $Y$ is the (smooth) geometric quotient $X^{ss}/G$. By the description of the unstable locus, we can assume $d_1+d_2> \sum_{i\geq 3}d_i$ without loss of generality. That is, $X^{ss}$ does not intersect the big diagonal $\{p_1=p_2\}\subset (\mathbb{P}^1)^n$. Call $V=0 \times \infty \times (\mathbb{P}^1)^{n-2}$ and consider $\mathbb{G}_m$ as the subgroup of $PGL_2$ given by
\begin{align}\label{torus each component}
\begin{bmatrix}
t & 0 \\
0 & t^{-1}
\end{bmatrix}.
\end{align}
Observe $\mathbb{G}_m$ acts on $V$, and the linearization $\mathcal{L}=\mathcal{O}_X(d_1,\ldots,d_n)$ restricts to a $\mathbb{G}_m$-linearization of $\mathcal{L}|_{V}=\mathcal{O}_V(d_3,\dots,d_n)$, which corresponds to the character $t\mapsto t^{d_1-d_2}$. By the stability condition, we see that every stable $G$-orbit intersects $V$. Further, $V\cap X^{ss}$ is precisely the semi-stable locus $V^{ss}=V^s$ for the $\mathbb{G}_m$-linearization of $\mathcal{L}|_{V}$. In fact, $Z=(0,\infty,z_3,\ldots,z_n)\in V$ is unstable if and only if there is some $I'\subset \{3,\ldots,n\}$ such that one of the following holds:
\begin{enumerate}
\renewcommand{\theenumi}{\alph{enumi}}
\item $d_2+\sum_{i\in I'}d_i>d_1+\sum_{i \notin I'}d_i$ and $z_i=\infty$ for all $i\in I'$, or
\item $d_2+\sum_{i\in I'}d_i<d_1+\sum_{i \notin I'}d_i$ and $z_i=0$ for all $i\notin I'$.
\end{enumerate}
Since $d_1+d_2> \sum_{i\geq 3}d_i$, this is the same stability condition for the $PGL_2$ action. Observe $\mathbb{G}_m$ acts freely on $V^{ss}$, and $V^{ss}\git \mathbb{G}_m$ is a geometric quotient.

We note that in fact the GIT quotients $X\git G$ and $V \git \mathbb{G}_m$ coincide. To see this, we can look at the coordinate rings of invariants. First, observe that $p_{i1}=x_iy_1-y_ix_1$ restricts to $x_i$ on $V$, while $p_{2i}$ restricts to $y_i$. Call $2\delta = d_1+d_2-\sum_{i\geq 3}d_i >0$. Then the restriction $\bigoplus_{k\geq 0} H^0(X,\mathcal{L}^{\otimes k})^{PGL_2} \rightarrow \bigoplus_{k\geq 0} H^0(V,\mathcal{L}|_{V}^{\otimes k})^{\mathbb{G}_m}$ is an isomorphism of graded rings, with inverse given in degree $k$ by
\begin{align*}
R(x_i,y_i) \mapsto p_{21}^{\delta k} R(p_{i1},p_{2i}),
\end{align*}
for a polynomial $R(x_3,y_3;\ldots;x_n,y_n)\in H^0(V,\mathcal{O}_V(kd_3,\ldots,kd_n))^{\mathbb{G}_m}$.

Now let the torus $T=(\mathbb{G}_m)^{n-2}$ act on $V=(\mathbb{P}^1)^{n-2}$, by (\ref{torus each component}) in each component. Then $V^{ss}=X^{ss}\cap V$ is invariant under the action of $T$. In fact, suppose $z=(0,\infty,z_3,\ldots,z_n) \in X^{us}$. Then $z_i=\alpha$ for every $i\in I$, for some $I$ such that $\sum_{i\notin I} d_i>\sum_{i\in I}d_i$. Since $d_1+d_2>\sum_{i\geq 3}d_i$, either $1 \in I$, in which case $\alpha=0$, or $2\in I$, in which case $\alpha=\infty$. But both $0$ and $\infty$ are fixed by $\mathbb{G}_m$, so $t\cdot z$ will still be unstable for any $t \in T$.

Further, $T$ acts on $V^{ss}$ with an open dense orbit, say $T\cdot (0,\infty,1,\ldots,1)$. We conclude that the $(n-3)$-dimensional torus $T/\mathbb{G}_m$ acts on $Y=V^{ss}/\mathbb{G}_m$ with an open dense orbit. Therefore $Y$ is a toric variety.
\end{proof}

Another proof of the previous lemma can be found in \cite[Theorem 2]{schmitt}, using the point of view of variation of GIT.

Now suppose $X^{us}$ has codimension $\geq 2$. We claim that any ample line bundle on $Y$ is the descent of an ample line bundle $\mathcal{O}_X(d'_1,\ldots,d'_n)$ on $X$ living in the same GIT chamber as $\mathcal{L}$, in the sense of \cite{vgit}.

\begin{lemma} \label{L'=L}
Suppose $X^{us}$ has codimension at least $2$ and let $L$ be an ample line bundle on $Y=X \git_\mathcal{L} PGL_2$, where $\mathcal{L}$ is such that $X^s=X^{ss}$. Then $L$ is the descent of an ample line bundle $\mathcal{L}'=\mathcal{O}_X(d'_1,\ldots,d'_n)$ such that $Y= X\git_{\mathcal{L}'} PGL_2$.
\end{lemma}

\begin{proof}
Since the action of $G$ is free on $X^{ss}=X^s$, by Kempf's descent lemma \cite[Theorem 2.3]{kempfs}, every $G$-equivariant line bundle on $X$ descends to a line bundle on $Y$, and in fact $\pi^*$ is an isomorphism from $\Pic Y$ to the $G$-equivariant Picard group $\Pic^G X^{ss} = \Pic \mathfrak{X}^{ss}$, with inverse $\mathcal{L}'\mapsto\pi_*(\mathcal{L}')^{G}$. Further, every $PGL_2$-linearized line bundle on $X^{ss}$ extends uniquely to a $PGL_2$-linearized line bundle on $X$ by the codimension hypothesis (see e.g. \cite[\S 7]{dolgachev-lectures}).

For any $v=(v_1,\ldots,v_n)\in \mathbb{Z}^n$, $PGL_2$ acts naturally on the global sections $H^0(X,\mathcal{O}_X(v))$. Let $R$ be the $\mathbb{Z}^n$-graded ring $R=\bigoplus_{v \in \mathbb{Z}^n} R_v$, where $R_v=H^0(X,\mathcal{O}_X(v))^{PGL_2}$. Notice $R_v=0$ if $\sum v_i$ is odd or if some $v_i<0$. If $\mathcal{L}=\mathcal{O}_X(w)$ is the linearization, then by definition $Y=X\git_{\mathcal{L}}PGL_2= \proj \bigoplus_{k \geq 0} R_{kw}$. From the previous observation, every $PGL_2$-linearized line bundle $\mathcal{O}_X(v)$ descends to a line bundle $\mathcal{L}_v$ on $Y$, and by the codimension hypothesis, $R_v=H^0(Y,\mathcal{L}_v)$. On the other hand, given a line bundle $L$ on $Y$, $L$ must be $\mathcal{L}_{w'}$ for some $w'\in \mathbb{Z}^n$. If $\mathcal{L}_{w'}$ is ample, then $Y=\proj \bigoplus_{k \geq 0}(Y,\mathcal{L}_{w'}^{\otimes k})=\proj\bigoplus_{k \geq 0}R_{kw'}$. From this we see that the $w'_i$ are nonnegative, and in fact $w'_i>0$ since $\dim Y=n-3$. That is, $\mathcal{L}'=\mathcal{O}_X(w')$ is ample and $Y=X\git_{\mathcal{L}'}PGL_2$.
\end{proof}

\begin{remark}
In the situation of Lemma \ref{L'=L}, we have $\codim X^{us} \geq 2$ and $\Pic Y=\Pic^GX$, so the quotient $X\git_{\mathcal{L}'}PGL_2$ can only be isomorphic to $Y=X\git_\mathcal{L}PGL_2$ if $\mathcal{L}$ and $\mathcal{L}'$ live in the same GIT chamber (cf. \cite[Proposition 5.1]{hassett}). Observe that such $\mathcal{L}'$, descending to an ample line bundle on $Y$, cannot be in one of the GIT walls, because the ample cone of $Y$ is open in $\Pic Y=\Pic^G X$. In particular, $\mathcal{L}'$ admits no strictly semi-stable locus.
\end{remark}

If we want to show vanishing for $H^i(Y,\Omega_Y^j\otimes L)$, by Corollary \ref{can use teleman} and Lemma \ref{F computes hypercohomology}, we need to compute $H^i(\mathfrak{X},\Lambda^j L_\mathfrak{X}\otimes \mathcal{L})=\mathcal{H}^i(F^\cdot)^G$, where $F^\cdot$ is given by (\ref{koszul F}). From Corollary \ref{almost bott}, we know $\mathcal{H}^i(F^\cdot)^G=0$ for $i\neq 0,j$, so it remains to show that the maps of $G$-invariant global sections $(H^0(X,\Omega_X \otimes \mathcal{L}) \otimes S^{j-1}\mathfrak{g}^\vee)^G \rightarrow (H^0(X,\mathcal{L})\otimes S^j\mathfrak{g}^\vee)^G$ are surjective. The following two propositions show this holds when $j=1$ and $2$.

\begin{proposition} \label{graphs for H^1}
Let $\mathcal{L}=\mathcal{O}_X(d_1,\ldots,d_n)$ be a linearization with no strictly semi-stable locus. The map $H^0(X,\Omega_X \otimes \mathcal{L})^G \rightarrow (H^0(X,\mathcal{L})\otimes \mathfrak{g}^\vee)^G$ is surjective.
\end{proposition}

\begin{proposition} \label{graphs for H^2}
Let $\mathcal{L}=\mathcal{O}_X(d_1,\ldots,d_n)$ be a linearization with no strictly semi-stable locus. The map $(H^0(X,\Omega_X \otimes \mathcal{L})\otimes \mathfrak{g}^\vee)^G \rightarrow (H^0(X,\mathcal{L})\otimes S^2\mathfrak{g}^\vee)^G$ is surjective.
\end{proposition}

In order to prove these two propositions, we will first investigate invariant global sections. Observe that for a given line bundle $\mathcal{O}_{(\mathbb{P}^1)^s}(l_1,\ldots,l_s)$ on $(\mathbb{P}^1)^s$, global sections can be written as $H^0((\mathbb{P}^1)^s,\mathcal{O}_{(\mathbb{P}^1)^s}(l_1,\ldots,l_s))=V_{l_1}\otimes \cdots \otimes V_{l_s}$, where $V_l$ is the irreducible $(l+1)$-dimensional representation of $\mathfrak{sl}_2$. We can also identify $V_l$ with the space of degree $l$ polynomials in two variables, $V_l=\langle x^l,x^{l-1}y,\ldots,y^l \rangle$, with the action given by $g\cdot p(x,y)=p(g^{-1}\cdot (x,y))$, for $g \in SL_2$. In particular, from Gelfand-MacPherson correspondence, the vector space $(V_{l_1}\otimes\cdots \otimes V_{l_s})^{PGL_2}$ can be identified with the elements of multi-degree $(l_1,\ldots,l_s)$ in the homogeneous coordinate ring of the Grassmannian $\Bbbk[p_{ik}]/(p_{ik}p_{rl}-p_{ir}p_{kl}+p_{il}p_{kr})$. 

\begin{remark} \label{V2=g=gv}
For $l=2$, write $V_2=\langle x_0^2,x_0y_0,y_0^2\rangle$. We have $\mathfrak{g}\cong V_2$ as $\mathfrak{g}$-representations, by identifying the bases $\{E,H,F\}$ and $\{y_0^2,2x_0y_0,-x_0^2\}$. If we further use the isomorphism of $\mathfrak{g}$-representations $\mathfrak{g}\cong \mathfrak{g}^\vee$, we get $\{X_0,Y_0,Z_0\}=\{x_0^2,x_0y_0,y_0^2\}$.
\end{remark}

Let us use this identification of $\mathfrak{g}^\vee \cong V_2$. The map $\Omega_X \rightarrow \mathfrak{g}^\vee$ is then determined by the $n$ sections $s_i=x_i^2Z_0-2x_iy_iY_0+y_i^2X_0=(x_0y_i-x_iy_0)^2 \in (H^0(X,\mathcal{O}_X(0,\ldots,2,\ldots,0))\otimes \mathfrak{g}^\vee)^G$, by taking $X_0=x_0^2$, $Y_0=x_0y_0$, $Z_0=y_0^2$, where $\{X_0,Y_0,Z_0\}$ is the basis of $\mathfrak{g}^\vee$ dual to $\{-E,H,F\}$, and $H^0(X,\mathcal{O}_X(0,\ldots,2,\ldots,0))=V_2=\langle x_i^2,x_iy_i,y_i^2\rangle$.

We see further that the symmetric powers $S^m \mathfrak{g}^\vee$ split canonically as $V_{2m} \oplus S^{m-2}\mathfrak{g}^\vee$ as $\mathfrak{g}$-representations, for $m\geq 2$. Indeed, let $\mathbb{P}^2=\mathbb{P}(\mathfrak{g})$, so that $S^m\mathfrak{g}^\vee = H^0(\mathbb{P}(\mathfrak{g}),\mathcal{O}_{\mathbb{P}^2}(m))$, and let $C=\mathbb{P}^1$ be the $G$-invariant conic in $\mathbb{P}^2$ defined by $X_0Z_0-Y_0^2=0$. The curve $C$ is given in coordinates by the rational normal curve embedding $(x^2:xy:y^2)$. Using the tautological short exact sequence and tensoring with $\mathcal{O}_{\mathbb{P}^2}(m)$, we get
\begin{align} \label{tautological on P2}
0 \rightarrow \mathcal{O}_{\mathbb{P}^2}(m-2) \rightarrow \mathcal{O}_{\mathbb{P}^2}(m) \rightarrow \mathcal{O}_C(2m)\rightarrow 0.
\end{align}

Taking global sections we get $0 \rightarrow S^{m-2} \mathfrak{g}^\vee \rightarrow S^m \mathfrak{g}^\vee \rightarrow V_{2m} \rightarrow 0$. By semisimplicity of $\mathfrak{sl}_2$, this splits in a unique way. Observe that the map $S^{m-2}\mathfrak{g}^\vee\rightarrow S^m\mathfrak{g}^\vee$ is multiplication by $X_0Z_0-Y_0^2$, while the map $S^m\mathfrak{g}^\vee\rightarrow V_{2m}$ sends precisely $\{X_0,Y_0,Z_0\}$ to $\{x_0^2,x_0y_0,y_0^2\}$, where we write $V_{2m}=\langle x_0^{2m},\ldots,y_0^{2m}\rangle$. Now consider again the complex $F^\cdot$ from (\ref{koszul F}), with its differentials $H^0(\Omega_X^{j-m+1}\otimes \mathcal{L})\otimes S^{m-1}\mathfrak{g}^\vee\rightarrow H^0(\Omega_X^{j-m}\otimes \mathcal{L})\otimes S^{m}\mathfrak{g}^\vee$. Using the splittings $S^r \mathfrak{g}^\vee = V_{2r} \oplus S^{r-2}\mathfrak{g}^\vee$, compose with the inclusion $V_{2m-2}\hookrightarrow S^{m-1}\mathfrak{g}^\vee$ and the projection $S^m\mathfrak{g}^\vee\rightarrow V_{2m}$ to get a map $H^0(\Omega_X^{j-m+1}\otimes \mathcal{L})\otimes V_{2m-2}\rightarrow H^0(\Omega_X^{j-m}\otimes \mathcal{L})\otimes V_{2m}$. That is, the map making the following diagram commute
\[
\begin{tikzcd}[cramped,sep=scriptsize]
H^0(\Omega_X^{j-m+1}\otimes \mathcal{L})\otimes S^{m-1}\mathfrak{g}^\vee \arrow{r}\arrow{d} &H^0(\Omega_X^{j-m}\otimes \mathcal{L})\otimes S^{m}\mathfrak{g}^\vee\arrow{d}\\
H^0(\Omega_X^{j-m+1}\otimes \mathcal{L})\otimes V_{2m-2}\arrow{r}\arrow[hook, bend right=33]{u} &H^0(\Omega_X^{j-m}\otimes \mathcal{L})\otimes V_{2m}\arrow[hook, bend right=33]{u}.
\end{tikzcd}
\]
This way we get a new complex
\begin{align} \label{F bar explicit}
\bar{F}^\cdot =\left[ 0\rightarrow H^0(X,\Omega^j_X \otimes \mathcal{L}) \rightarrow H^0(X,\Omega_X^{j-1} \otimes \mathcal{L})\otimes V_2 \rightarrow \cdots \rightarrow H^0(X,\mathcal{L})\otimes V_{2j}\rightarrow 0 \right],
\end{align}
which we can think of as a ``partial'' version of $F^\cdot$. Observe that, by commutativity of the diagram above, $\bar{F}^\cdot$ is indeed a chain complex.

By the discussion above, the differential maps in $\bar{F}^\cdot$ correspond to multiplication by $s_i=(x_0y_i-x_iy_0)^2$, where $x_i,y_i$ are coordinates in the $i$-th component, and $x_0,y_0$ correspond to the terms $V_{2m}=\langle x_0^{2m},\ldots,y_0^{2m}\rangle$. In what follows next, we will study the complex $\bar{F}^\cdot$, and then we will see that from this we can get back some information about the original complex $F^\cdot$ from (\ref{koszul F}).

\subsection{Computations in $(\mathbb{P}^1)^{n+1}$}

Now for $j>0$, we consider the diagonal action of $PGL_2$ on $(\mathbb{P}^1)^{n+1}=\mathbb{P}^1 \times X$. Using coordinates $x_0,y_0;x_i,y_i$, take $$s_i=(x_0y_i-x_iy_0)^2 \in H^0(\mathbb{P}^1\times X,\mathcal{O}(2;0,\ldots,2,\ldots,0))^{PGL_2}$$ for $i=1,\ldots,n$. We choose the polarization $\mathcal{V}=\mathcal{O}(2j;d_1,\ldots,d_n) = \mathcal{O}_{\mathbb{P}^1}(2j) \boxtimes \mathcal{L}$.

\begin{proposition} \label{Koszul on XxP^1}
Suppose $\mathcal{L}=\mathcal{O}_X(d_1,\ldots,d_n)$ is a polarization on $X$ with no strictly semi-stable locus, and let $\mathcal{V}=\mathcal{O}_{\mathbb{P}^1}(2j) \boxtimes \mathcal{L}$ as above. Let $M$ be the scheme-theoretic intersection $\bigcap (s_i=0) \subset \mathbb{P}^1 \times X$. Then $M$ is a local complete intersection and $H^0(M,\mathcal{V}|_M)$ has no $PGL_2$-invariants.
\end{proposition}

\begin{proof}
Write $D_i=(x_0y_i-x_iy_0=0)$ so that $M= \bigcap 2D_i$, while $\bigcap D_i$ is the small diagonal $\mathbb{P}^1 \subset (\mathbb{P}^1)^{n+1}$. Then $M=\bigcap 2D_i$ is a local complete intersection, having codimension $n$. For a reduced divisor $D \subset V$, we have a tautological short exact sequence
\begin{align} \label{tautological}
0 \rightarrow \mathcal{O}_D(-D) \rightarrow \mathcal{O}_{2D} \rightarrow \mathcal{O}_D \rightarrow 0.
\end{align}

\begin{claim}
For every $0 \leq m \leq n$ and every $I \subset \{1,\ldots,m\}$, the sheaf
\begin{align} \label{no invs}
\mathcal{O}_{\bigcap_{i \leq m}D_i\cap \bigcap_{i>m}2D_i} (-\sum_{i \in I} D_i) \otimes \mathcal{V}
\end{align}
has no $PGL_2$-invariant global sections.
\end{claim}
Given the claim, the proposition is proved by taking $m=0$ in (\ref{no invs}). To prove the claim we use (\ref{tautological}) on $\bigcap_{i \leq m+1} D_i \cap \bigcap_{i>m+1} 2D_i \subset \bigcap_{i \leq m} D_i \cap \bigcap_{i>m+1} 2D_i$, to get
\begin{align*}
0 \rightarrow \mathcal{O}_{\bigcap_{i \leq m+1} D_i \cap \bigcap_{i>m+1} 2D_i}(-D_{m+1}) \rightarrow \mathcal{O}_{\bigcap_{i \leq m} D_i \cap \bigcap_{i>m} 2D_i} \rightarrow \mathcal{O}_{\bigcap_{i \leq m+1} D_i \cap \bigcap_{i>m+1} 2D_i} \rightarrow 0.
\end{align*}
Now tensor with $\mathcal{V}(-\sum_{i \in I} D_i)$ and take $PGL_2$-invariant global sections. The claim will then be proved if we show $\mathcal{O}_{\bigcap_{i \leq m+1} D_i \cap \bigcap_{i>m+1} 2D_i}(-\sum_{i \in I'}D_i)$ has no invariant global sections for every $I'\subset \{1,\ldots,m+1\}$. That is, the claim is true for $m$ if it is true for $m+1$. Therefore, we can do induction on $n-m$, so that all we need to show is that
\begin{align*}
H^0(\mathcal{O}_{\bigcap_{i \leq n} D_i}(-\sum_{i \in I} D_i)\otimes \mathcal{V})^G=0
\end{align*}
for any $I \subset \{1,\ldots,n\}$.

Recall $\bigcap_{i \leq n} D_i = \mathbb{P}^1$ is the small diagonal and $\mathcal{O}_{(\mathbb{P}^1)^{n+1}}(-D_i)=\mathcal{O}(-1;0,\ldots,-1,\ldots,0)$, so that $\mathcal{O}_{\bigcap_{i \leq n} D_i}(-\sum_{i \in I} D_i)\otimes \mathcal{V} = \mathcal{O}_{\mathbb{P}^1}(2j+\sum_{i \leq n} d_i -2|I|)$. The $PGL_2$-invariant global sections of this sheaf are homogeneous polynomials in $x$ and $y$ of degree $2j+\sum_{i \leq n} d_i -2|I|$ that are restrictions to the small diagonal of polynomials in $p_{ik}=x_iy_k-x_ky_i$. Of course, any such polynomial will restrict to $0$ in the diagonal, unless it has degree $0$. But $2j+\sum_{i \leq n} d_i-2|I|$ cannot be zero. This follows from the following claim.
\begin{claim}
$\sum_{i=1}^nd_i \geq 2n$. In particular $2j+\sum_{i=1}^n d_i-2|I| >0$ for every $I\subset \{1,\ldots,n\}$.
\end{claim}
Let us prove this claim. Without loss of generality, we may assume $d_1\leq \ldots \leq d_n$. Choose $0\leq m \leq n$ such that $d_1=\ldots=d_m=1$, $d_{m+2}\geq 2$ and $m$ has the same parity as $n$. Observe that, since $\sum d_i$ is even and $\mathcal{L}$ has no strictly semi-stable locus, as a consequence of Remark \ref{no sss} we must have
\begin{align}\label{numerical d_i}
d_n+d_{n-2}+\ldots+d_{m+2}>d_{n-1}+\ldots+d_{m+3}+d_{m+1}+m.
\end{align}
In fact, if $d_n+d_{n-2}+\ldots+d_{m+2}-(d_{n-1}+\ldots+d_{m+3}+d_{m+1})=r \leq m$, we would have $d_n+d_{n-2}+\ldots+d_{m+2}=(d_{n-1}+\ldots+d_{m+3}+d_{m+1})+(d_1+\ldots+d_r)$, and then writing each of the remaining $d_{r+1}=d_{r+2}=\ldots=d_m=1$ at either side of this equation we would get $\sum_{i\notin I}d_i=\sum_{i\in I}d_i$, where $I=\{1,\ldots,r\}\cup\{r+1,r+3,\ldots\}\cup\{m+1,m+3,\ldots,n-1\}$, a contradiction. In particular, from (\ref{numerical d_i}) we have $d_n>d_{m+1}+m$. Then $\sum d_i = m+d_{m+1}+\sum_{i=m+2}^{n-1}d_i+d_n > 2(m+d_{m+1})+\sum_{i=m+2}^{n-1}d_i$. Since $d_{m+1} \geq 1$ and $d_{n-1}\geq \ldots \geq d_{m+2} \geq 2$, this is at least $2m+2+2(n-m-2)=2n-2$. Thus $\sum_{i=1}^n d_i>2n-2$, so in fact $\sum_{i=1}^n d_i \geq 2n$. This completes the proof.
\end{proof}

\begin{corollary} \label{exactness on V2j}
With the same hypotheses, $\mathcal{H}^i(\bar{F^\cdot})^G=0$ for $i>1$.
\end{corollary}

\begin{proof}
Since $M=\bigcap (s_i=0)$ is a local complete intersection, the augmented Koszul complex determined by $s_1,\ldots,s_n$,
\begin{align*}
0\rightarrow\mathcal{O}(-\sum_{i=1}^n 2D_i)\rightarrow\cdots\rightarrow\bigoplus\mathcal{O}(-2D_i)\rightarrow\mathcal{O}\rightarrow \mathcal{O}_M\rightarrow 0
\end{align*}
is acyclic, and so is the complex
\begin{align}\label{explicit Koszul complex on XxP^1}
K^\cdot=\left[ 0\rightarrow\mathcal{V}(-\sum 2 D_{i}) \rightarrow \cdots \rightarrow \bigoplus \mathcal{V}(-2 D_{i}) \rightarrow \mathcal{V} \rightarrow \mathcal{V} \otimes \mathcal{O}_M\rightarrow 0 \right].
\end{align}
We consider this complex having nonzero terms in degrees $-n$ to $1$. This means that for $-n\leq p\leq 0$, the term $K^p$ is precisely
\begin{align*}
\bigoplus_{|I|=-p}\mathcal{V}(-\sum_{i\in I}2D_i)=\mathcal{O}_{\mathbb{P}^1}(2j+2p)\boxtimes(\Omega_X^{-p}\otimes \mathcal{L}).
\end{align*}
Take the spectral sequence $E_1^{pq}=H^q((\mathbb{P}^1)^{n+1},K^p)$, which converges to $H^{p+q}((\mathbb{P}^1)^{n+1},K^\cdot)=0$. We get
\begin{align}\label{spectral sequence K}
H^q((\mathbb{P}^1)^{n+1},K^p)= 
\begin{cases}
H^0(X,\Omega_X^{-p}\otimes \mathcal{L}) \otimes V_{2j+2p} & \text{if} \quad q=0,\: -j\leq p\leq 0 \\
H^0(X,\Omega_X^{-p}\otimes \mathcal{L}) \otimes H^1(\mathbb{P}^1,\mathcal{O}_{\mathbb{P}^1}(2j+2p)) & \text{if} \quad q=1,\: p<-j \\
H^q(M,\mathcal{V}|_M) & \text{if} \quad p=1 \\
0 & \text{otherwise}
\end{cases}
\end{align}
and the sequence has the following shape
\[
\begin{tikzcd}[cramped,sep=scriptsize]
\cdots\arrow[r] &E_1^{-j-2,1}\arrow[r] &E_1^{-j-1,1}\arrow[r] &0\arrow[r]          &\cdots &\cdots\arrow[r] &0 \arrow[r] &H^1(M,\mathcal{V}|_M)\\
                &\cdots\arrow[r]       & 0\arrow[r]           &E_1^{-j,0}\arrow[r] &E_1^{-j+1,0}\arrow[r] &\cdots \arrow[r] &E_1^{0,0} \arrow[r]  &H^0(M,\mathcal{V}|_M).
\end{tikzcd}
\]

The complex $\bar{F}^\cdot$ from (\ref{F bar explicit}) is the same as the (shifted) naive truncation of $E_1^{\cdot,0}[-j]$ obtained by omitting the last term $H^0(M,\mathcal{V}|_M)$ of $E_1^{\cdot,0}$, since the differentials are determined precisely by the sections $s_i$.

We see that for $q=0$ and $p>-j+1$, the sequence degenerates at $E_2$ and we get $0=H^{i-j}(\mathbb{P}^1\times X,K^\cdot)=\mathcal{H}^i(\bar{F}^\cdot)$, for $1< i< j$ (even before taking invariants). Further, since $H^0(M,\mathcal{V}|_M)^{G}=0$ by the previous proposition, the complex of $G$-invariants $(E_1^{\cdot,0}[-j])^G$ is precisely $(\bar{F}^\cdot)^G$, so $\mathcal{H}^j((\bar{F}^\cdot)^G)=0$ too, that is, $\mathcal{H}^j(\bar{F}^\cdot)^G=0$.
\end{proof}

\subsection{Directed graphs as invariant sections}

Given a $G$-linearized ample line bundle $\mathcal{L}=\mathcal{O}_X(d_1,\ldots,d_n)$ on $X$, let us use the identifications $H^0(X,\mathcal{L})=V_{d_1}\otimes \cdots \otimes V_{d_n}$, and $H^0(X,\Omega_X \otimes \mathcal{L})=\bigoplus_{i=1}^n V_{d_1}\otimes\cdots\otimes V_{d_i-2}\otimes\cdots \otimes V_{d_n}$. We also use $\mathfrak{g}^\vee=V_2$ and $S^2\mathfrak{g}^\vee=V_0\oplus V_4$ as $\mathfrak{g}$-representations. Then to show Propositions \ref{graphs for H^1} and \ref{graphs for H^2}, we need to investigate the maps 
\begin{align} \label{map for H^1}
\bigoplus_{i=1}^n (V_{d_1}\otimes\cdots\otimes V_{d_i-2}\otimes\cdots \otimes V_{d_n})^{G}\overset{t_1}\longrightarrow (V_2\otimes V_{d_1}\otimes \cdots \otimes V_{d_n})^{G}
\end{align}
and
\begin{align}\label{map for H^2}
\bigoplus_{i=1}^n (V_2\otimes V_{d_1}\otimes\cdots\otimes V_{d_i-2}\otimes\cdots \otimes V_{d_n})^{G}\overset{t_2}\longrightarrow (V_4\otimes V_{d_1}\otimes \cdots \otimes V_{d_n})^{G} \oplus (V_0\otimes V_{d_1}\otimes \cdots \otimes V_{d_n})^{G}.
\end{align}

Namely, we need to show that both maps (\ref{map for H^1}) and (\ref{map for H^2}) are surjective. In view of Gelfand-MacPherson correspondence, we will work with these invariants using the language of graphs (as in \cite{vakil+3} and \cite{vakil+3-II}).

\begin{notation}
Let $J$ be a directed graph with vertices $V(J)$ and edges $E(J)$. Let $w \in V(J)$ be a vertex. By $\deg (w)$ we mean the number of edges touching $w$. We say that two vertices $w$ and $v$ are adjacent if there is an edge between them. An edge going from $w$ to $v$ will be denoted by $w\rightarrow v$.
\end{notation}

A directed graph $J$ can be represented by a $2 \times m$ tableau, where $m= |E(J)|$. A diagram
\begin{align*}
\begin{bmatrix}
a_1 & \ldots & a_m \\
b_1 & \ldots & b_m
\end{bmatrix}
\end{align*}
represents the graph with edges $a_i\rightarrow b_i$.

\begin{definition}
Let $l=(l_1,\ldots,l_r) \in \mathbb{Z}^r_{\geq 0}$, with $\sum l_i$ even. We call $\mathfrak{F}_l$ the free vector space generated by directed graphs $J$ having $r$ vertices, say $V(J)=\{w_1,\ldots,w_r\}$, with degrees $\deg(w_i)=l_i$. We denote by $\mathfrak{F}'_l$ the quotient of $\mathfrak{F}_l$ by the following relations:
\begin{enumerate}
\renewcommand{\theenumi}{\alph{enumi}}
\item If $K$ is obtained from $J$ by reversing the direction of one edge, then $K=-J$. In particular, any graph having a self-loop is equal to zero in $\mathfrak{F}'_l$.
\item The relation $J=H+K$, whenever $H$ and $K$ are obtained by replacing a $2\times 2$ submatrix as follows:
\begin{align*}
\begin{bmatrix}
\cdots & a & \cdots & b & \cdots \\
\cdots & c & \cdots & d & \cdots
\end{bmatrix}
=
\begin{bmatrix}
\cdots & a & \cdots & c & \cdots \\
\cdots & b & \cdots & d & \cdots
\end{bmatrix}
+
\begin{bmatrix}
\cdots & a & \cdots & b & \cdots \\
\cdots & d & \cdots & c & \cdots
\end{bmatrix}
\end{align*}
\end{enumerate}
\end{definition}

We observe that the space $\mathfrak{F}'_l$ is exactly identified with the vector space of invariants $(V_{l_1} \otimes \cdots \otimes V_{l_r})^{PGL_2}$. A Pl\"ucker minor $p_{ik}=x_iy_k-x_ky_i$ corresponds to an edge $w_i \rightarrow w_k$, and the relations defining $\mathfrak{F}'_l$ are precisely $p_{ik}=-p_{ki}$ and the Pl\"ucker relations. Pl\"ucker relation is drawn as follows:

%now Plucker relations
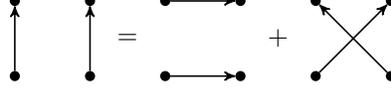
\begin{center}
\begin{tikzpicture}[>=stealth',semithick,inner sep=0.4mm]
\tikzstyle{every label}=[font=\footnotesize]
\node (A) at ( 0,0) [circle,draw,fill] {};
\node (B) at ( 0,1) [circle,draw,fill] {};
\node (C) at ( 1,0) [circle,draw,fill] {};
\node (D) at ( 1,1) [circle,draw,fill] {};
%\path (A) edge node{} (B);
\draw [->] (A.north) -- (B.south);
\draw [->] (C.north) -- (D.south);
\node (e) at (1.5,0.5) [circle] {$=$};
\node (E) at ( 2,0) [circle,draw,fill] {};
\node (F) at ( 2,1) [circle,draw,fill] {};
\node (G) at ( 3,0) [circle,draw,fill] {};
\node (H) at ( 3,1) [circle,draw,fill] {};
\draw [->] (E.east) -- (G.west);
\draw [->] (F.east) -- (H.west);
\node (f) at (3.5,0.5) [circle] {$+$};
\node (I) at ( 4,0) [circle,draw,fill] {};
\node (J) at ( 4,1) [circle,draw,fill] {};
\node (K) at ( 5,0) [circle,draw,fill] {};
\node (L) at ( 5,1) [circle,draw,fill] {};
\draw [->] (I.north east) -- (L.south west);
\draw [->] (K.north west) -- (J.south east);
\end{tikzpicture}
\end{center}

If $\sum l_i$ is odd or if one $l_i<0$, we just set $\mathfrak{F}'_l=0$. Then for fixed $r$, we can put all the spaces $\mathfrak{F}'_l$ together in a $\mathbb{Z}^r$-graded ring $\mathfrak{F}'=\bigoplus_{l \in \mathbb{Z}^r}\mathfrak{F}'_l$. This is the same construction as the ring $R$ defined in the proof of Lemma \ref{L'=L}. In this language, the product of two graphs $J_1$ and $J_2$ consists of a graph having edges $E(J_1J_2)=E(J_1)\cup E(J_2)$. An element $J\in\mathfrak{F}'_l$ is a graph if it is written as a product of Pl\"ucker minors $p_{ik}$. In general, an element of $\mathfrak{F}'_l$ is a polynomial in $p_{ik}$, that is, a linear combination of graphs.

We will be mostly interested in the spaces $\mathfrak{F}'_l$ when $l=(2m,d_1,\ldots,d_n)$. For the graphs in $\mathfrak{F}'_l$, we label the $n+1$ vertices as $w_0,w_1,\ldots,w_n$, so that $\deg{w_0}=2m$, $\deg{w_i}=d_i$ for $i\geq 1$. We call $V(J)_0$ the set of vertices adjacent to $w_0$, and for $w_i$ we call $e(w_0,w_i)$ the number of edges between $w_0$ and $w_i$.

\begin{definition}
Let $l=(2m,d_1,\ldots,d_n)$ and let $J$ be a directed graph in $\mathfrak{F}_l$, as above. A $2$-coloring of $J$ is an assignment $c:V(J)-\{w_0\} \rightarrow \{0,1\}$ such that $c(a) \neq c(b)$ for every two adjacent vertices $a$ and $b$, and also $\sum_{w_i \in c^{-1}(0)} e(w_0,w_i)=\sum_{w_i \in c^{-1}(1)}e(w_0,w_i)=m$.
\end{definition}

\begin{example}
The graph given by
\begin{align*}
\begin{bmatrix}
w_0 & w_0 & w_0 & w_0 & w_1 & w_2 & w_3 \\
w_1 & w_1 & w_2 & w_3 & w_2 & w_4 & w_4
\end{bmatrix}
\end{align*}
admits a $2$-coloring:

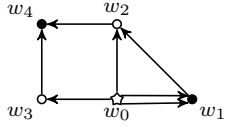
\begin{center}
\tikzstyle{every label}=[font=\footnotesize]
\begin{tikzpicture}[>=stealth',semithick,inner sep=0.4mm]
\node (A) at ( 0,0) [star,label=below:$w_0$,draw] {};
\node (B) at ( 0,1) [circle,label=above:$w_2$,draw] {};
\node (C) at ( -1,0) [circle,label=below left:$w_3$,draw] {};
\node (D) at ( -1,1) [circle,label=above left:$w_4$,draw,fill] {};
\node (E) at (1,0) [circle,label=below right:$w_1$,draw,fill]{};
\draw [->] (A.north) -- (B.south);
\draw [->] (A.west) -- (C.east);
\draw [->] (B.west) -- (D.east);
\draw [->] (C.north) -- (D.south);
\draw [->] (A.north east) -- (E.north west);
\draw [->] (A.south east) -- (E.south west);
\draw [->] (E.north west) -- (B.south east);
\end{tikzpicture}
\end{center}

\end{example}

If $m=1$, we can think of a $2$-coloring as a bipartition of the graph obtained by deleting $w_0$ and replacing the edges coming from it by an edge joining the two vertices $w_{i_1},w_{i_2} \in V(J)_0$. In this bipartition $w_{i_1}$ and $w_{i_2}$ must be in different blocks. In particular, if a graph $J\in \mathfrak{F}'_{(2,d_1,\ldots,d_n)}$ has a double edge coming from $w_0$, that is, if $w_{i_1}=w_{i_2}$, then $J$ cannot admit a $2$-coloring. For coloring purposes, the directions of the edges are irrelevant.

\begin{remark}
Suppose $\mathcal{L}=\mathcal{O}_X(d_1,\ldots,d_n)$ is such that $X^{ss}=X^s$. Then no graph $J \in \mathfrak{F}_{(2m,d_1,\ldots,d_n)}$ admits a $2$-coloring. Indeed, if $J$ had a $2$-coloring, then we can call $I=\{i \mid c(w_i)=0\} \subset \{1,\ldots,n\}$, so that $\sum_{i \in I} d_i = \sum_{i \notin I} d_i$.
\end{remark}

\begin{lemma}
The image of the map $t_1$ from (\ref{map for H^1}) consists of the vector subspace generated by graphs having a double edge coming from $w_0$.
\end{lemma}

\begin{proof}
By the explicit description of $\bar{F}^\cdot$ in (\ref{F bar explicit}), we know the maps $V_{d_1}\otimes\cdots\otimes V_{d_i-2}\otimes\cdots \otimes V_{d_n}\rightarrow V_2\otimes V_{d_1}\otimes \cdots \otimes V_{d_n}$ are given by multiplication by $s_i=(x_0y_i-x_iy_0)^2$. Taking invariants we get maps $\mathfrak{F}'_{(d_1,\ldots,d_i-2,\ldots,d_n)}\rightarrow\mathfrak{F}'_{(2,d_1,\ldots,d_n)}$. We identify $\mathfrak{F}'_{(d_1,\ldots,d_i-2,\ldots,d_n)}=\mathfrak{F}'_{(0,d_1,\ldots,d_i-2,\ldots,d_n)}$ by adding an extra vertex $w_0$ with $\deg{w_0}=0$. Then multiplication of a graph $J$ by $s_i$ corresponds to adding two extra edges to $J$, both going from $w_0$ to $w_i$.
\end{proof}

\begin{notation}
Let $l=(2m,d_1,\ldots,d_n)$. A cycle is a sequence of vertices $w_{i_1},\ldots,w_{i_r}$ such that each $w_{i_k}$ is adjacent to $w_{i_{k+1}}$, and $w_{i_r}$ is adjacent to $w_{i_1}$. We say that the cycle is central if it involves the vertex $w_0$. We call $r$ the length of the cycle. A subgraph $C$ determined by the cycle $w_{i_1},\ldots,w_{i_r}$ will be denoted by $(w_{i_1},\ldots,w_{i_r})$ if the signs of the edges are given by
\begin{align*}
C=
\begin{bmatrix}
w_{i_1} & w_{i_2} & \cdots& w_{i_{r-1}} & w_{i_r} \\
w_{i_2} & w_{i_3} & \cdots& w_{i_r} & w_{i_1}
\end{bmatrix}.
\end{align*}
\end{notation}

For a cycle we do not require that all the vertices $w_{i_k}$ be different. We observe that rotating the indices $i_1,\ldots,i_r$ does not change the cycle, while reversing an arrow switches the sign. 

\begin{remark} \label{odd cycle or 2-color}
Let $l=(0,d_1,\ldots,d_n)$ and $J$ a graph in $\mathfrak{F}'_l$. It is a well-known fact that $J$ admits a $2$-coloring if and only if it does not contain a cycle of odd length. This fact is sometimes referred to as K\H{o}nig's Theorem.
\end{remark}

\begin{lemma} \label{even cycles}
Suppose $J \in \mathfrak{F}'_{(2,d_1,\ldots,d_n)}$ is a graph having a central cycle of even length. Then $J$ is in the image of the map $t_1$ from (\ref{map for H^1}).
\end{lemma}

\begin{proof}
We can assume $w_0,\ldots,w_{r}$ is a cycle in $J$, where $r$ is odd. Let $J_0$ be the subgraph given by the cycle, $J_0=(w_0,\ldots,w_r)$, so that $J$ is a multiple of $J_0$, say $J=J_0H$. It suffices to show that $J_0$ can be written as a linear combination of graphs having a double edge from $w_0$. Consider the Pl\"ucker relation
\begin{align*}
\begin{bmatrix}
w_r & w_1 \\
w_0 & w_2
\end{bmatrix}
=
\begin{bmatrix}
w_r & w_0 \\
w_1 & w_2
\end{bmatrix}
+
\begin{bmatrix}
w_r & w_1 \\
w_2 & w_0
\end{bmatrix}
\end{align*}
or

%even cycles (hexagons)
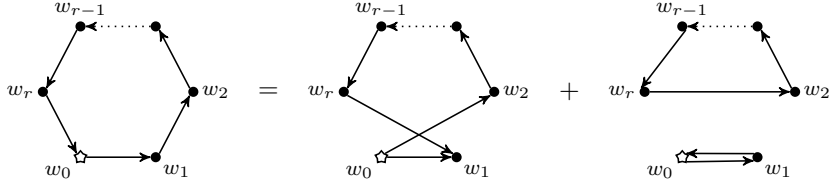
\begin{center}
\begin{tikzpicture}[>=stealth',semithick,inner sep=0.4mm]
\tikzstyle{every label}=[font=\footnotesize]
\node (A) at ( 0,-0.87) [label=below left:$w_0$,star,draw] {};
\node (B) at ( 1,-0.87) [label=below right:$w_1$,circle,draw,fill] {};
\node (C) at ( 1.5,0) [label=right:$w_2$,circle,draw,fill] {};
\node (D) at ( 1,0.87) [circle,draw,fill] {};
\node (E) at ( 0,0.87) [label=above:$w_{r-1}$,circle,draw,fill] {};
\node (F) at ( -0.5,0) [label=left:$w_r$,circle,draw,fill] {};
\draw [->] (A.east) -- (B.west);
\draw [->] (B.north east) -- (C.south west);
\draw [->] (C.north west) -- (D.south east);
\draw [->,dotted] (D.west) -- (E.east);
\draw [->] (E.south west) -- (F.north east);
\draw [->] (F.south east) -- (A.north west);
\node (a) at (2.5,0) [circle] {$=$};

\node (A') at ( 4,-0.87) [label=below left:$w_0$,star,draw] {};
\node (B') at ( 5,-0.87) [label=below right:$w_1$,circle,draw,fill] {};
\node (C') at ( 5.5,0) [label=right:$w_2$,circle,draw,fill] {};
\node (D') at ( 5,0.87) [circle,draw,fill] {};
\node (E') at ( 4,0.87) [label=above:$w_{r-1}$,circle,draw,fill] {};
\node (F') at ( 3.5,0) [label=left:$w_r$,circle,draw,fill] {};
\draw [->] (A'.east) -- (B'.west);
\draw [->] (F'.south east) -- (B'.north west);
\draw [->] (C'.north west) -- (D'.south east);
\draw [->,dotted] (D'.west) -- (E'.east);
\draw [->] (E'.south west) -- (F'.north east);
\draw [->] (A'.north east) -- (C'.south west);

\node (a) at (6.5,0) [circle] {$+$};
\node (A'') at ( 8,-0.87) [label=below left:$w_0$,star,draw] {};
\node (B'') at ( 9,-0.87) [label=below right:$w_1$,circle,draw,fill] {};
\node (C'') at ( 9.5,0) [label=right:$w_2$,circle,draw,fill] {};
\node (D'') at ( 9,0.87) [circle,draw,fill] {};
\node (E'') at ( 8,0.87) [label=above:$w_{r-1}$,circle,draw,fill] {};
\node (F'') at ( 7.5,0) [label=left:$w_r$,circle,draw,fill] {};
\draw [->] (A''.south east) -- (B''.south west);
\draw [->] (B''.north west) -- (A''.north east);
\draw [->] (C''.north west) -- (D''.south east);
\draw [->,dotted] (D''.west) -- (E''.east);
\draw [->] (E''.south east) -- (F''.north west);
\draw [->] (F''.east) -- (C''.west);

\end{tikzpicture}
\end{center}
so that we have $J_0=J_1+J'_1$, where $J'_1$ has a double edge between $w_0$ and $w_1$. Then $J \equiv J_1 H\mod \im(t_1)$. On the other hand, $J_1$ is equivalent to the cycle $-(w_1,w_0,w_2,\ldots,w_r)$. Similarly, given a cycle $J_i=(-1)^i(w_1,\ldots,w_{i},w_0,w_{i+1},\ldots,w_r)$, we use the Pl\"ucker relation on
\begin{align*}
\begin{bmatrix}
w_{i} & w_{i+1} \\
w_{0} & w_{i+2}
\end{bmatrix}
\end{align*}
to obtain $J \equiv (-1)^{i+1}J_{i+1}H \mod \im(t_1)$. We conclude $J \equiv (-1)^r J_rH \equiv -J \mod \im (t_1)$, since $r$ is odd and $J_r=(w_1,\ldots,w_r,w_0)=J_0$. Then $J \in \im (t_1)$, as desired.
\end{proof}

Now we have the tools to show Proposition \ref{graphs for H^1}.

\begin{proof}[Proof of Proposition \ref{graphs for H^1}]
Let $K^\cdot$ be the complex (\ref{explicit Koszul complex on XxP^1}), where $\mathcal{V}=\mathcal{O}_{\mathbb{P}^1}(2)\boxtimes \mathcal{L}=\mathcal{O}(2;d_1,\ldots,d_n)$, and consider the spectral sequence $E_1^{pq}=H^q((\mathbb{P}^1)^{n+1},K^p)$ from (\ref{spectral sequence K}):
\[
\begin{tikzcd}[cramped,sep=scriptsize]
\cdots\arrow[r] &E_1^{-3,0}\arrow[r] &E_1^{-2,0}\arrow[r] &0\arrow[r]          &\cdots &\\
                &\cdots\arrow[r]        &0\arrow[r]          &E_1^{-1,0}\arrow[r] &E_1^{0,0}\arrow[r] &H^0(M,\mathcal{V}|_M).
\end{tikzcd}
\]
Since the complex $K^{\cdot}$ is acyclic (cf. the proof of Corollary \ref{exactness on V2j}), this spectral sequence converges to zero. The restriction of $d_1^{-1,0}$ to invariant sections is $t_1$. We need to show it is surjective. The second page of the spectral sequence has the following shape:
\[
\begin{tikzcd}[cramped,sep=scriptsize]
\cdots &E_2^{-3,0}\arrow{rrd}{d_2} &E_2^{-2,0}\arrow{rrd}{d_2} &0          &\cdots &\\
                &\cdots        &0          &E_2^{-1,0} &E_2^{0,0} &H^0(M,\mathcal{V}|_M)/\im(d_1^{0,0})
\end{tikzcd}
\]

We want to describe the restriction of the map $d_2^{-2,0}$ to invariant sections, that is $(d_2^{-2,0})^G:(E_2^{-2,0})^G\rightarrow (E_2^{0,0})^G$. Observe $(E_2^{0,0})^G=(E_1^{0,0})^G/\im(t_1)$ since $H^0(M,\mathcal{V}|_M)^G=0$ by Proposition \ref{Koszul on XxP^1}. The whole sequence degenerates at $E_3$, so $d_2^{-2,0}$ must be an isomorphism, in particular surjective. Therefore, any $J \in (E_1^{0,0})^G$ can be written as a sum $J'+J''$, where $J'\in \im(t_1)$ and $J''\in \im((d_2^{-2,0})^G)$.

The map $d_2^{-2,0}$ is obtained by doing a resoultion of $K^{-2}\rightarrow K^{-1}\rightarrow K^0$ that computes cohomologies of $K^p$, and then chasing the diagram. Since, for each $q$ and $p$, $H^q(\mathbb{P}^1\times X,\mathcal{O}_{\mathbb{P}^1}(2+2p)\boxtimes (\Omega_X^{-p}\otimes \mathcal{L}))=H^q(\mathbb{P}^1,\mathcal{O}_{\mathbb{P}^1}(2+2p))\otimes H^0(X,\Omega^{-p}\otimes \mathcal{L})$, it suffices to use resolutions of $\mathcal{O}_{\mathbb{P}^1}(2+2p)$ and then tensor with $H^0(X,\Omega_X^{-p}\otimes L)$, for $p=-2,-1,0$. We use the usual \v{C}ech resolution, given by $S_{x_0}\times S_{y_0}\rightarrow S_{x_0y_0}$, where $S=\Bbbk [x_0,y_0]$ and that, when restricted to rational functions of a given degree $l$, it computes the cohomologies of $\mathcal{O}_{\mathbb{P}^1}(l)$ (see e.g. \cite[\S III.5.1]{hartshorne}). We have
\[
\begin{tikzcd}[cramped,column sep=1.18em,row sep=scriptsize]
(S_{x_0y_0})_{-2}\otimes H^0(X,\Omega_X^{2}\otimes \mathcal{L})\arrow{r}{f} &(S_{x_0y_0})_{0}\otimes H^0(X,\Omega_X\otimes \mathcal{L})\arrow{r}{f} &(S_{x_0y_0})_{2}\otimes H^0(X,\mathcal{L}) \\
(S_{x_0}\times S_{y_0})_{-2}\otimes H^0(X,\Omega_X^{2}\otimes \mathcal{L})\arrow{u}{h}\arrow{r}{f} &(S_{x_0}\times S_{y_0})_{0}\otimes H^0(X,\Omega_X\otimes \mathcal{L})\arrow{u}{h}\arrow{r}{f} &(S_{x_0}\times S_{y_0})_{2}\otimes H^0(X,\mathcal{L}).\arrow{u}{h}
\end{tikzcd}
\]

Write $H^0(X,\Omega_X^{2}\otimes \mathcal{L})=\bigoplus_{l>k} V_{d_1}\otimes\cdots\otimes V_{d_k-2}\otimes\cdots\otimes V_{d_l-2}\otimes\cdots\otimes V_{d_n}$ and $H^0(X,\Omega_X\otimes \mathcal{L})=\bigoplus_{i} V_{d_1}\otimes\cdots\otimes V_{d_i-2}\otimes\cdots\otimes V_{n}$, so that the map $K^{-2}\rightarrow K^{-1}$ is $\sum f^{kl}$, where each $f^{kl}$ is multiplication by $s_l$ onto the $k$-th component and multiplication by $-s_k$ onto the $l$-th component. The map $t:K^{-1}\rightarrow K^0$ is $\sum t^i$ where $t^i$ is multiplication by $s_i$. Recall $s_i=p_{0i}^2=(x_0y_i-x_iy_0)^2$.

Let $u\in (S_{x_0y_0})_{-2}\otimes (V_{d_1}\otimes\cdots\otimes V_{d_k-2}\otimes\cdots\otimes V_{d_l-2}\otimes\cdots\otimes V_{d_n})$. We then have $f^{kl}(u)=(\ldots,s_ku,\ldots,-s_lu,\ldots)$, with zeros in the remaining coordinates. Write
\begin{align*}
u=\frac{P}{x_0^m}+\frac{Q}{y_0^m}+\frac{R}{x_0y_0}
\end{align*}
for some polynomials $P$, $Q$, $R$, whose homogeneous degrees with respect to $x_0,y_0$ are $m-2$, $m-2$, $0$, respectively. Then $s_ku=h(v)$ for some $v \in (S_{x_0}\times S_{y_0})_{0}\otimes(V_{d_1}\otimes\cdots\otimes V_{d_l-2}\otimes\cdots\otimes V_{n})$, and we can choose
\begin{align*}
v=\left( Q \frac{s_k}{x_0^m}+R\frac{x_k^2y_0}{x_0} -R x_ky_k , -P\frac{s_k}{y_0^m}-R\frac{x_0y_k^2}{y_0}+Rx_ky_k \right).
\end{align*}
Similarly, we find $v'$ such that $-s_lu=h(v')$. To find $d_2^{-2,0}(u)$ we then need to compute $f(v)+f(v')=s_lv+s_kv'$. We get $s_lv+s_kv'=(b,b)$, where
\begin{align*}
b=R\left( \frac{y_0}{x_0}(s_lx_k^2-s_kx_l^2)+s_kx_ly_l-s_lx_ky_k \right).
\end{align*}
Simplifying we get $b=R(x_0y_l-x_ly_0)(x_0y_k-x_ky_0)(x_ly_k-x_ky_l)=p_{0l}p_{0k}p_{lk}R$. Now $H^0(\mathbb{P}^1\times X,\mathcal{V})$ is identified with the diagonal of $(S_{x_0}\times S_{y_0})_{2}\otimes H^0(X,\mathcal{L})$ so, if we call $\bar{v}\in E_2^{-2,0}$ the class represented by $v$, we have $d_2^{-2,0}(\bar{v})=p_{0l}p_{0k}p_{lk}R$, a multiple of $p_{0l}p_{0k}p_{lk}$. If $\bar{v}$ was invariant, then $d_2^{-2,0}(\bar{v})$ is a linear combination of graphs having $p_{0l}p_{0k}p_{lk}$ as a subgraph, that is, graphs that have a central cycle $w_0,w_l,w_k$ of length $3$.

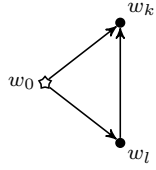
\begin{center}
\begin{tikzpicture}[>=stealth',semithick,inner sep=0.4mm]
\tikzstyle{every label}=[font=\footnotesize]
\node (A) at ( 0,0) [label=left:$w_0$,star,draw] {};
\node (B) at ( 1,-0.8) [label=below right:$w_l$,circle,draw,fill] {};
\node (C) at ( 1,0.8) [label=above right:$w_k$,circle,draw,fill] {};
\draw [->] (A.south east) -- (B.north west);
\draw [->] (A.north east) -- (C.south west);
\draw [->] (B.north) -- (C.south);
\end{tikzpicture}
\end{center}

Therefore, modulo $\im(t_1)$, every $J \in \mathfrak{F}'_{(2,d_1,\ldots,d_n)}$ is a linear combination of graphs of the form $p_{0l}p_{0k}p_{lk}R$. Then it suffices to show that all such graphs are in the image of $t_1$. Given $J=p_{0l}p_{0k}p_{lk}R$, call $J'\in \mathfrak{F}'_{(0,d_1,\ldots,d_n)}$ the graph obtained by replacing the two edges $w_0\rightarrow w_k$, $w_0\rightarrow w_l$ by an extra $w_l\rightarrow w_k$, that is, $J'=p_{lk}^2R$. Observe that a $2$-coloring of $J'$ would need to have $c(w_k)\neq c(w_l)$, so it would give a $2$-coloring on $J$. Since $\mathcal{L}$ has no strictly semi-stable locus, $J$ and $J'$ do not admit a $2$-coloring and by Remark \ref{odd cycle or 2-color} $J'$ must contain some odd cycle, say $(w_{i_1},\ldots,w_{i_r})$. Note that any two vertices that are adjacent on $J'$ are adjacent on $J$ too, so in fact $(w_{i_1},\ldots,w_{i_r})$ is an odd cycle in $J$, that is not central. Apply the Pl\"ucker relation
\begin{align*}
\begin{bmatrix}
w_l & w_{i_1} \\
w_k & w_{i_2}
\end{bmatrix}
=
\begin{bmatrix}
w_l & w_k \\
w_{i_1} & w_{i_2}
\end{bmatrix}
+
\begin{bmatrix}
w_l & w_{i_1} \\
w_{i_2} & w_k
\end{bmatrix}
\end{align*}
or

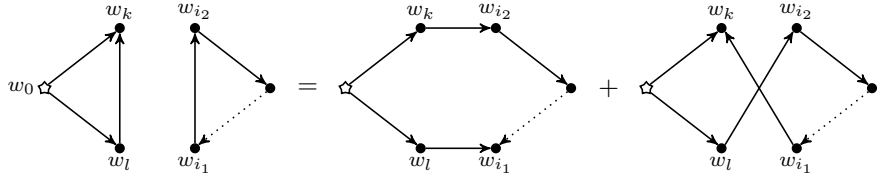
\begin{center}
\begin{tikzpicture}[>=stealth',semithick,inner sep=0.4mm]
\tikzstyle{every label}=[font=\footnotesize]
\node (A) at ( 0,0) [label=left:$w_0$,star,draw] {};
\node (C) at ( 1,0.8) [label=above:$w_k$,circle,draw,fill] {};
\node (B) at ( 1,-0.8) [label=below:$w_l$,circle,draw,fill] {};
\draw [->] (A.north east) -- (C.south west);
\draw [->] (A.south east) -- (B.north west);
\draw [->] (B.north) -- (C.south);

\node (D) at ( 2,-0.8) [label=below:$w_{i_1}$,circle,draw,fill] {};
\node (E) at ( 2,0.8) [label=above:$w_{i_2}$,circle,draw,fill] {};
\node (F) at ( 3,0) [circle,draw,fill] {};
\draw [->] (D.north) -- (E.south);
\draw [->] (E.south east) -- (F.north west);
\draw [->,dotted] (F.south west) -- (D.north east);

\node (a) at (3.5,0)[circle]{=};

\node (A') at ( 4,0) [star,draw] {};
\node (C') at ( 5,0.8) [label=above:$w_k$,circle,draw,fill] {};
\node (B') at ( 5,-0.8) [label=below:$w_l$,circle,draw,fill] {};
\node (D') at ( 6,-0.8) [label=below:$w_{i_1}$,circle,draw,fill] {};
\node (E') at ( 6,0.8) [label=above:$w_{i_2}$,circle,draw,fill] {};
\node (F') at ( 7,0) [circle,draw,fill] {};

\draw [->] (A'.north east) -- (C'.south west);
\draw [->] (A'.south east) -- (B'.north west);
\draw [->] (B'.east) -- (D'.west);
\draw [->] (C'.east) -- (E'.west);
\draw [->] (E'.south east) -- (F'.north west);
\draw [->,dotted] (F'.south west) -- (D'.north east);

\node (b) at (7.5,0)[circle]{+};

\node (A'') at ( 8,0) [star,draw] {};
\node (C'') at ( 9,0.8) [label=above:$w_k$,circle,draw,fill] {};
\node (B'') at ( 9,-0.8) [label=below:$w_l$,circle,draw,fill] {};
\node (D'') at ( 10,-0.8) [label=below:$w_{i_1}$,circle,draw,fill] {};
\node (E'') at ( 10,0.8) [label=above:$w_{i_2}$,circle,draw,fill] {};
\node (F'') at ( 11,0) [circle,draw,fill] {};

\draw [->] (A''.north east) -- (C''.south west);
\draw [->] (A''.south east) -- (B''.north west);
\draw [->] (B''.north east) -- (E''.south west);
\draw [->] (D''.north west) -- (C''.south east);
\draw [->] (E''.south east) -- (F''.north west);
\draw [->,dotted] (F''.south west) -- (D''.north east);

\end{tikzpicture}
\end{center}

to get $J=H+K$, where $H$ contains the cycle $w_0,w_k,w_{i_2},\ldots,w_{i_r},w_{i_1},w_l$ and $K$ contains the cycle $w_0,w_l,w_{i_2},\ldots,w_{i_r},w_{i_1},w_k$, both of even length $r+3$. By Lemma \ref{even cycles}, $J\in \im(t_1)$. We conclude $t_1$ is surjective.
\end{proof}

Next, we investigate the map $t_2$ from (\ref{map for H^2}). According to the splitting $S^2\mathfrak{g}^\vee=V_4\oplus V_0$ we write $t_2=(t,t')$, and further $t=\sum t^i$, $t'=\sum t'^i$, where $t^i:\mathfrak{F}'_{(2,d_1,\ldots,d_i-2,\ldots,d_n)}\rightarrow\mathfrak{F}'_{(4,d_1,\ldots,d_n)}$ and $t'^i:\mathfrak{F}'_{(2,d_1,\ldots,d_i-2,\ldots,d_n)}\rightarrow\mathfrak{F}'_{(0,d_1,\ldots,d_n)}$. Let us describe these maps in terms of graphs with vertices $\{w_0,w_1,\ldots,w_n\}$.

\begin{lemma}
Let $J$ be a graph in $\mathfrak{F}'_{(2,d_1,\ldots,d_n)}$. Write $J$ as a polynomial, $J=p_{0k}p_{0l}H$. Then $t^i(J)=p_{0i}^2J$, while $t'^i(J)=\frac{2}{3}p_{ik}p_{il}H$. That is, $t^i$ adds a double edge $w_0 \rightarrow w_i$ to the graph while, up to a constant, $t'^i$ replaces the edges $w_0\rightarrow w_k$, $w_0\rightarrow w_l$ by $w_i\rightarrow w_k$ and $w_i\rightarrow w_l$:

%splitting of the map t_2
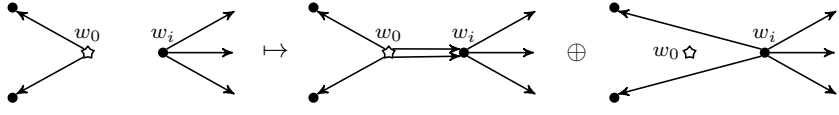
\begin{center}
\begin{tikzpicture}[>=stealth',semithick,inner sep=0.4mm]
\tikzstyle{every label}=[font=\footnotesize]
\node (A) at ( 0,0) [label=above:$w_0$,star,draw] {};
\node (B) at ( -1,0.6) [circle,draw,fill] {};
\node (C) at ( -1,-0.6) [circle,draw,fill] {};
\node (D) at ( 1,0) [label={[shift={(0.0,0.028)}]$w_i$},circle,draw,fill] {};
\node (E) at ( 2,0.6) [circle] {};
\node (F) at ( 2,0) [circle] {};
\node (G) at ( 2,-0.6) [circle] {};
\draw [->] (A.north west) -- (B.south east);
\draw [->] (A.south west) -- (C.north east);
\draw [->] (D.north east) -- (E.south west);
\draw [->] (D.east) -- (F.west);
\draw [->] (D.south east) -- (G.north west);

\node (a) at (2.5,0) [circle] {$\mapsto$};

\node (A') at ( 4,0) [label=above:$w_0$,star,draw] {};
\node (B') at ( 3,0.6) [circle,draw,fill] {};
\node (C') at ( 3,-0.6) [circle,draw,fill] {};
\node (D') at ( 5,0) [label={[shift={(0.0,0.028)}]$w_i$},circle,draw,fill] {};
\node (E') at ( 6,0.6) [circle] {};
\node (F') at ( 6,0) [circle] {};
\node (G') at ( 6,-0.6) [circle] {};
\draw [->] (A'.north west) -- (B'.south east);
\draw [->] (A'.south west) -- (C'.north east);
\draw [->] (D'.north east) -- (E'.south west);
\draw [->] (D'.east) -- (F'.west);
\draw [->] (D'.south east) -- (G'.north west);
\draw [->] (A'.north east) -- (D'.north west);
\draw [->] (A'.south east) -- (D'.south west);

\node (b) at (6.5,0) [circle] {$\oplus$};

\node (A'') at ( 8,0) [label=left:$w_0$,star,draw] {};
\node (B'') at ( 7,0.6) [circle,draw,fill] {};
\node (C'') at ( 7,-0.6) [circle,draw,fill] {};
\node (D'') at ( 9,0) [label={[shift={(0.0,0.028)}]$w_i$},circle,draw,fill] {};
\node (E'') at ( 10,0.6) [circle] {};
\node (F'') at ( 10,0) [circle] {};
\node (G'') at ( 10,-0.6) [circle] {};
\draw [->] (D''.north west) -- (B''.south east);
\draw [->] (D''.south west) -- (C''.north east);
\draw [->] (D''.north east) -- (E''.south west);
\draw [->] (D''.east) -- (F''.west);
\draw [->] (D''.south east) -- (G''.north west);

\end{tikzpicture}
\end{center}

\end{lemma}

\begin{proof}
By the explicit description of $\bar{F}^\cdot$ from (\ref{F bar explicit}), we know $t^i$ is multiplication by $(x_0y_i-x_iy_0)^2=p_{0i}^2$, which corresponds to adding two edges, both from $w_0$ to $w_i$.

Consider the splitting $S^2\mathfrak{g}^\vee=V_4\oplus V_0$ obtained from (\ref{tautological on P2}). Here $V_0$ is the one-dimensional vector space with the trivial action. The projection $\pi:S^2\mathfrak{g}\rightarrow V_0$ is the unique $\mathfrak{g}$-equivariant map that satisfies $\pi \circ \imath=\Id_{V_0}$, where $\imath:V_0\hookrightarrow S^2\mathfrak{g}^\vee$ is the inclusion from (\ref{tautological on P2}), namely, $\imath$ is multiplication by $X_0Z_0-Y_0^2$. We find $\pi$ explicitly, and it is defined as follows: for $P=\alpha Y_0^2 + \beta X_0Z_0 + \ldots \in S^2\mathfrak{g}^\vee$, $\pi(P)=\frac{1}{3}(2\beta-\alpha)$. It is easy to check that $\pi$ is indeed a $\mathfrak{g}$-equivariant map: observe, for instance, that $E\cdot(X_0Y_0)=F\cdot(Y_0Z_0)=-X_0Z_0-2Y_0^2$, and $\pi(-X_0Z_0-2Y_0^2)=0$. Indeed, this together with the fact that all monomials other than $X_0Z_0$ and $Y_0^2$ map to zero ensures that $\pi(g\cdot P)=0$ $\forall g\in \mathfrak{g}$, $P\in S^2\mathfrak{g}^\vee$, so $\pi$ is a map of representations. Further, we see that $\pi(X_0Z_0-Y_0^2)=1$, and then $\pi \circ \imath=\Id_{V_0}$. By uniqueness, $\pi$ must be the desired map.

Now look at $t'=\sum t'^i$. Each $t'^i$ is given by multiplication by $x_i^2Z_0-2x_iy_iY_0+y_i^2X_0$ followed by the projection $\pi$ from $S^2\mathfrak{g}^\vee$. Suppose $J \in (V_2\otimes V_{d_1}\otimes\cdots\otimes V_{d_i-2}\otimes\cdots \otimes V_{d_n})^{PGL_2}=\mathfrak{F}'_{(2,d_1,\ldots,d_{i}-2,\ldots,d_n)}$ is a directed graph, written as $J=(Ax_0^2+Bx_0y_0+Cy_0^2)H=(AX_0+BY_0+CZ_0)H$ for some polynomials $A,B,C$ and $H$. Multiplying by $x_i^2Z_0-2x_iy_iY_0+y_i^2X_0$ and looking at the terms involving $X_0Z_0$ and $Y_0^2$, we find $t'^i(J)=\frac{2}{3}(Ax_i^2+Bx_iy_i+Cy_i^2)H$. Now, since $J$ is actually a $PGL_2$-invariant section, it has to be of the form $J=(x_0y_k-x_ky_0)(x_0y_l-x_ly_0)H$. That is, $J$ is a graph where the two edges coming from $w_0$ are $w_0\rightarrow w_k$ and $w_0\rightarrow w_l$ (up to sign). Then we have $A=y_ky_l$, $B=-(y_kx_l+y_lx_k)$, $C=x_kx_l$ and we compute
\begin{align*}
t'^i(J)=\frac{2}{3}(x_iy_l-x_ly_i)(x_iy_k-x_ky_i)H.
\end{align*}
That is, up to multiplication by $2/3$, the map $t'^i$ precisely erases the edges $w_0\rightarrow w_k$, $w_0\rightarrow w_l$, and replaces them by $w_i\rightarrow w_k$, $w_i\rightarrow w_l$.
\end{proof}

Since multiplying everything in $\mathfrak{F}'_{(0,d_1,\ldots,d_n)}$ by a constant does not change the image of the map $t_2$, from now on we just ignore the constant $2/3$ appearing in $t'$. Now we can prove Proposition \ref{graphs for H^2}.

\begin{proof}[Proof of Proposition \ref{graphs for H^2}]
Write $t_2=(t,t')$, according to the decomposition in (\ref{map for H^2}). By Corollary \ref{exactness on V2j}, $t$ is surjective. Then it suffices to show that, for any graph $H\in \mathfrak{F}'_{(0,d_1,\ldots,d_n)}=\mathfrak{F}'_{(d_1,\ldots,d_n)}$, we have $(0,H) \in \im(t_2)$.

\begin{step}
Let $J$ be a graph in $\mathfrak{F}'_{(4,d_1,\ldots,d_n)}$ having a subgraph $B$ of the form
\begin{align*}
B_{1,2,3}=
\begin{bmatrix}
w_0 & w_1 & w_2 & w_0 & w_0 \\
w_1 & w_2 & w_0 & w_3 & w_3
\end{bmatrix}.
\end{align*}
That is, $B_{1,2,3}$ has a cycle $(w_0,w_1,w_2)$ and a double edge between $w_0$ and $w_3$ (and similarly, $B_{i_1,i_2,i_3}$ denotes a permutation of indices in the expression above).

%the flask
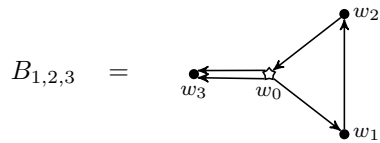
\begin{center}
\begin{tikzpicture}[>=stealth',semithick,inner sep=0.4mm]
\tikzstyle{every label}=[font=\footnotesize]
\node (a) at (-3,0) [circle]{$B_{1,2,3}$};
\node (b) at (-2,0) [circle]{$=$};
\node (A) at ( 0,0) [label={[shift={(0.0,-0.48)}]$w_0$},star,draw] {};
\node (B) at ( 1,-0.8) [label=right:$w_1$,circle,draw,fill] {};
\node (C) at ( 1,0.8) [label=right:$w_2$,circle,draw,fill] {};
\node (D) at ( -1,0) [label={[shift={(0.0,-0.44)}]$w_3$},circle,draw,fill] {};
\draw [->] (A.south east) -- (B.north west);
\draw [->] (B.north) -- (C.south);
\draw [->] (C.south west) -- (A.north east);
\draw [->] (A.north west) -- (D.north east);
\draw [->] (A.south west) -- (D.south east);
\end{tikzpicture}
\end{center}
Then we show $(J,0)\in \im(t_2)$, or in other words, $J \in t_2(\ker t')$. For this, write $J=B_{1,2,3}H$ and let $P=(w_0,w_1,w_2) \in \mathfrak{F}'_{(2,2,2,0,\ldots,0)}$ and $P'=(w_0,w_3,w_1) \in \mathfrak{F}'_{(2,2,0,2,\ldots,0)}$. We see that $t_2(PH-P'H)=(B_{1,2,3}H-B_{3,1,2}H,0)$, that is, $B_{1,2,3}H \equiv B_{3,1,2}H \mod t_2(\ker t')$. 

%relation of flasks
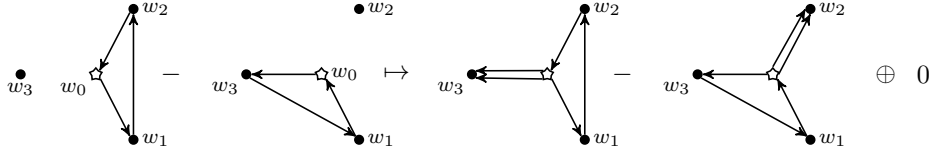
\begin{center}
\begin{tikzpicture}[>=stealth',semithick,inner sep=0.4mm]
\tikzstyle{every label}=[font=\footnotesize]

\node (A) at ( -0.5,0) [label=below left:$w_0$,star,draw] {};
\node (B) at ( 0,-0.87) [label=right:$w_1$,circle,draw,fill] {};
\node (C) at ( 0,0.87) [label=right:$w_2$,circle,draw,fill] {};
\node (D) at ( -1.5,0) [label=below:$w_3$,circle,draw,fill] {};
\draw [->] (A.south east) -- (B.north west);
\draw [->] (B.north) -- (C.south);
\draw [->] (C.south west) -- (A.north east);

\node (a) at (0.5,0) [circle]{$-$};

\node (A') at ( 2.5,0) [label=right:$w_0$,star,draw] {};
\node (B') at ( 3,-0.87) [label=right:$w_1$,circle,draw,fill] {};
\node (C') at ( 3,0.87) [label=right:$w_2$,circle,draw,fill] {};
\node (D') at ( 1.5,0) [label=below left:$w_3$,circle,draw,fill] {};
\draw [->] (A'.west) -- (D'.east);
\draw [->] (D'.south east) -- (B'.north west);
\draw [->] (B'.north west) -- (A'.south east);

\node (b) at (3.5,0) [circle]{$\mapsto$};

\node (A'') at ( 5.5,0) [star,draw] {};
\node (B'') at ( 6,-0.87) [label=right:$w_1$,circle,draw,fill] {};
\node (C'') at ( 6,0.87) [label=right:$w_2$,circle,draw,fill] {};
\node (D'') at ( 4.5,0) [label=below left:$w_3$,circle,draw,fill] {};
\draw [->] (A''.south east) -- (B''.north west);
\draw [->] (B''.north) -- (C''.south);
\draw [->] (C''.south west) -- (A''.north east);
\draw [->] (A''.north west) -- (D''.north east);
\draw [->] (A''.south west) -- (D''.south east);

\node (b) at (6.5,0) [circle]{$-$};

\node (A''') at ( 8.5,0) [star,draw] {};
\node (B''') at ( 9,-0.87) [label=right:$w_1$,circle,draw,fill] {};
\node (C''') at ( 9,0.87) [label=right:$w_2$,circle,draw,fill] {};
\node (D''') at ( 7.5,0) [label=below left:$w_3$,circle,draw,fill] {};
\draw [->] (A'''.west) -- (D'''.east);
\draw [->] (D'''.south east) -- (B'''.north west);
\draw [->] (B'''.north west) -- (A'''.south east);
\draw [->] (A'''.north) -- (C'''.west);
\draw [->] (A'''.east) -- (C'''.south);

\node (c) at (10,0) [circle]{$\oplus$};
\node (d) at (10.5,0)[circle]{$0$};

\end{tikzpicture}
\end{center}

On the other hand, take $B_{1,2,3}$ and apply the Pl\"ucker relation to the edges
\begin{align*}
\begin{bmatrix}
w_0 & w_1 \\
w_3 & w_2 
\end{bmatrix}
\end{align*}
to obtain $B_{1,2,3}=B_{3,2,1}+B_{1,3,2}$. 

%Plucker to flask
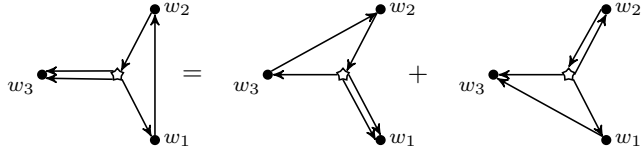
\begin{center}
\begin{tikzpicture}[>=stealth',semithick,inner sep=0.4mm]
\tikzstyle{every label}=[font=\footnotesize]

\node (A) at ( -0.5,0) [star,draw] {};
\node (B) at ( 0,-0.87) [label=right:$w_1$,circle,draw,fill] {};
\node (C) at ( 0,0.87) [label=right:$w_2$,circle,draw,fill] {};
\node (D) at ( -1.5,0) [label=below left:$w_3$,circle,draw,fill] {};
\draw [->] (A.south east) -- (B.north west);
\draw [->] (B.north) -- (C.south);
\draw [->] (C.south west) -- (A.north east);
\draw [->] (A.north west) -- (D.north east);
\draw [->] (A.south west) -- (D.south east);

\node (a) at (0.5,0) [circle]{$=$};

\node (A') at ( 2.5,0) [star,draw] {};
\node (B') at ( 3,-0.87) [label=right:$w_1$,circle,draw,fill] {};
\node (C') at ( 3,0.87) [label=right:$w_2$,circle,draw,fill] {};
\node (D') at ( 1.5,0) [label=below left:$w_3$,circle,draw,fill] {};
\draw [->] (A'.south) -- (B'.west);
\draw [->] (A'.east) -- (B'.north);
\draw [->] (A'.west) -- (D'.east);
\draw [->] (D'.north east) -- (C'.south west);
\draw [->] (C'.south west) -- (A'.north east);

\node (b) at (3.5,0) [circle]{$+$};

\node (A'') at ( 5.5,0) [star,draw] {};
\node (B'') at ( 6,-0.87) [label=right:$w_1$,circle,draw,fill] {};
\node (C'') at ( 6,0.87) [label=right:$w_2$,circle,draw,fill] {};
\node (D'') at ( 4.5,0) [label=below left:$w_3$,circle,draw,fill] {};
\draw [->] (A''.east) -- (C''.south);
\draw [->] (C''.west) -- (A''.north);
\draw [->] (A''.south east) -- (B''.north west);
\draw [->] (B''.north west) -- (D''.south east);
\draw [->] (A''.west) -- (D''.east);

\end{tikzpicture}
\end{center}

Also, we know $B_{1,2,3}=-B_{2,1,3}$ by reversing the arrows. Combining all these, we get that $B_{1,2,3}H\equiv B_{3,2,1}H+B_{1,3,2}H \equiv 2B_{3,2,1}H \equiv -2B_{2,3,1}H \equiv -2B_{1,2,3}H \mod t_2(\ker t')$. Thus we obtain $3B_{1,2,3}H \in t_2(\ker t')$, so $(J,0)\in \im(t_2)$.
\end{step}

\begin{step}
Let $J$ be a graph in $\mathfrak{F}'_{(4,d_1,\ldots,d_n)}$ having a subgraph $C$ of the form
\begin{align*}
C_{1,\ldots,r}=
\begin{bmatrix}
w_0 & w_1 & w_2 & w_0 & w_3 & \cdots & w_r \\
w_1 & w_2 & w_0 & w_3 & w_4 & \cdots & w_0
\end{bmatrix}
\end{align*}
for $r$ odd. That is, $C_{1,\ldots,r}$ has cycles $(w_0,w_1,w_2)$ and $(w_0,w_3,\ldots,w_r)$.

%the fish
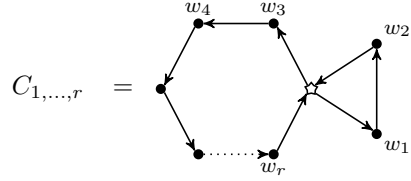
\begin{center}
\begin{tikzpicture}[>=stealth',semithick,inner sep=0.4mm]
\tikzstyle{every label}=[font=\footnotesize]
\node (a) at (-3.5,0) [circle]{$C_{1,\ldots,r}$};
\node (b) at (-2.5,0) [circle]{$=$};
\node (A) at ( 0,0) [star,draw] {};
\node (B) at ( 0.87,-0.6) [label=below right:$w_1$,circle,draw,fill] {};
\node (C) at ( 0.87,0.6) [label=above right:$w_2$,circle,draw,fill] {};
\node (D) at ( -0.5,0.87) [label=above:$w_3$,circle,draw,fill] {};
\node (E) at ( -1.5,0.87) [label=above:$w_4$,circle,draw,fill] {};
\node (F) at ( -2,0) [circle,draw,fill] {};
\node (G) at ( -1.5,-0.87) [circle,draw,fill] {};
\node (H) at ( -0.5,-0.87) [label=below:$w_r$,circle,draw,fill] {};
\draw [->] (A.south east) -- (B.north west);
\draw [->] (B.north) -- (C.south);
\draw [->] (C.south west) -- (A.north east);
\draw [->] (A.north west) -- (D.south east);
\draw [->] (D.west) -- (E.east);
\draw [->] (E.south west) -- (F.north east);
\draw [->] (F.south east) -- (G.north west);
\draw [->,dotted] (G.east) -- (H.west);
\draw [->] (H.north east) -- (A.south west);
\end{tikzpicture}
\end{center}
Then we see $(J,0) \in \im(t_2)$. Indeed, by Lemma \ref{even cycles}, the even cycle $(w_0,w_3,\ldots,w_r)$ can be written as a sum of graphs having double edges coming from $w_0$. Using this, $C_{1,\ldots,r}$ is written as a sum of graphs containing subgraphs of the form $B$ from Step 1. By Step 1, we get $J\in t_2(\ker t')$.
\end{step}

\begin{step}
Let $J$ be a graph in $\mathfrak{F}'_{(4,d_1,\ldots,d_n)}$ having a subgraph $B$ of the form
\begin{align} \label{balloon}
B_{1,\ldots,r}=
\begin{bmatrix}
w_0 & w_1 & \cdots & w_{r-1} & w_0 & w_0 \\
w_1 & w_2 & \cdots & w_0 & w_r & w_r
\end{bmatrix}
\end{align}
with $r$ odd. That is, $B_{1,\ldots,r}$ has an odd cycle $(w_0,\ldots,w_{r-1})$ and a double edge between $w_0$ and $w_r$.

%the balloon with pentagons
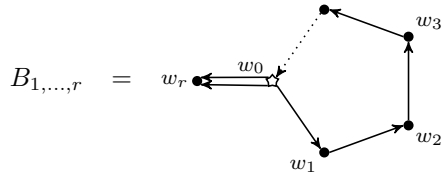
\begin{center}
\begin{tikzpicture}[>=stealth',semithick,inner sep=0.4mm]
\tikzstyle{every label}=[font=\footnotesize]
\node (a) at ( -3,0) [circle] {$=$};
\node (b) at ( -4,0) [circle] {$B_{1,\ldots,r}$};
\node (F) at ( -2,0) [label=left:$w_r$,circle,draw,fill] {};
\node (A) at ( -0.31,-0.95) [label=below left:$w_1$,circle,draw,fill] {};
\node (B) at ( 0.81,-0.59) [label=below right:$w_2$,circle,draw,fill] {};
\node (D) at ( 0.81,0.59) [label=above right:$w_3$,circle,draw,fill] {};
\node (E) at ( -0.31,0.95) [circle,draw,fill] {};
\node (H) at (-1,0) [label=above left:$w_0$,star,draw]{};
\draw [->] (H.south east) -- (A.north west);
\draw [->] (A.east) -- (B.west);
\draw [->] (B.north) -- (D.south);
\draw [->] (D.west) -- (E.east);
\draw [->,dotted] (E.south west) -- (H.north east);
\draw [->] (H.south west) -- (F.south east);
\draw [->] (H.north west) -- (F.north east);
\end{tikzpicture}
\end{center}
We show $(J,0)\in \im(t_2)$. If $r=3$, this is Step 1. Now suppose this is true for $r-2$. Write $J=C_{1,\ldots,r}H$ and do the Pl\"ucker relation to the edges
\begin{align*}
\begin{bmatrix}
w_0 & w_1\\
w_r & w_2
\end{bmatrix}
\end{align*}
to obtain $B_{1,\ldots,r}=B_{r,2,\ldots,r-1,1}+C$, where $C$ is a graph of the form given in Step 2.

%plucker to balloon, with pentagons
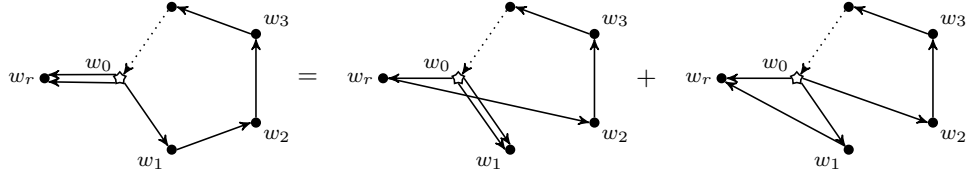
\begin{center}
\begin{tikzpicture}[>=stealth',semithick,inner sep=0.4mm]
\tikzstyle{every label}=[font=\footnotesize]

\node (F) at ( -2,0) [label=left:$w_r$,circle,draw,fill] {};
\node (A) at ( -0.31,-0.95) [label=below left:$w_1$,circle,draw,fill] {};
\node (B) at ( 0.81,-0.59) [label=below right:$w_2$,circle,draw,fill] {};
\node (D) at ( 0.81,0.59) [label=above right:$w_3$,circle,draw,fill] {};
\node (E) at ( -0.31,0.95) [circle,draw,fill] {};
\node (H) at (-1,0) [label=above left:$w_0$,star,draw]{};
\draw [->] (H.south east) -- (A.north west);
\draw [->] (A.east) -- (B.west);
\draw [->] (B.north) -- (D.south);
\draw [->] (D.west) -- (E.east);
\draw [->,dotted] (E.south west) -- (H.north east);
\draw [->] (H.south west) -- (F.south east);
\draw [->] (H.north west) -- (F.north east);

\node (a) at (1.5,0) [circle]{$=$};

\node (F') at ( 2.5,0) [label=left:$w_r$,circle,draw,fill] {};
\node (A') at ( -0.31+4.5,-0.95) [label=below left:$w_1$,circle,draw,fill] {};
\node (B') at ( 0.81+4.5,-0.59) [label=below right:$w_2$,circle,draw,fill] {};
\node (D') at ( 0.81+4.5,0.59) [label=above right:$w_3$,circle,draw,fill] {};
\node (E') at ( -0.31+4.5,0.95) [circle,draw,fill] {};
\node (H') at (-1+4.5,0) [label=above left:$w_0$,star,draw]{};
\draw [->] (H'.east) -- (A'.north);
\draw [->] (H'.south) -- (A'.west);
\draw [->] (B'.north) -- (D'.south);
\draw [->] (D'.west) -- (E'.east);
\draw [->,dotted] (E'.south west) -- (H'.north east);
\draw [->] (H'.west) -- (F'.east);
\draw [->] (F'.east) -- (B'.west);

\node (b) at (6,0) [circle]{$+$};

\node (F'') at ( 7,0) [label=left:$w_r$,circle,draw,fill] {};
\node (A'') at ( -0.31+9,-0.95) [label=below left:$w_1$,circle,draw,fill] {};
\node (B'') at ( 0.81+9,-0.59) [label=below right:$w_2$,circle,draw,fill] {};
\node (D'') at ( 0.81+9,0.59) [label=above right:$w_3$,circle,draw,fill] {};
\node (E'') at ( -0.31+9,0.95) [circle,draw,fill] {};
\node (H'') at (-1+9,0) [label=above left:$w_0$,star,draw]{};
\draw [->] (H''.east) -- (B''.west);
\draw [->] (B''.north) -- (D''.south);
\draw [->] (D''.west) -- (E''.east);
\draw [->,dotted] (E''.south west) -- (H''.north east);
\draw [->] (H''.west) -- (F''.east);
\draw [->] (A''.north west) -- (F''.south east);
\draw [->] (H''.south east) -- (A''.north west);

\end{tikzpicture}
\end{center}
Therefore, $B_{1,\ldots,r}H \equiv B_{r,2,\ldots,r-1,1}H \mod t_2(\ker t')$. On the other hand, if we use Pl\"ucker on 
\begin{align*}
\begin{bmatrix}
w_1 & w_3\\
w_2 & w_4
\end{bmatrix}
\end{align*}
we get $B_{1,\ldots,r}=-B_{1,3,2,4,\ldots,r}+B'$, where $B'$ is a graph containing $B_{1,4,5,\ldots,r}$ as a subgraph.

%other plucker on balloon, with pentagons
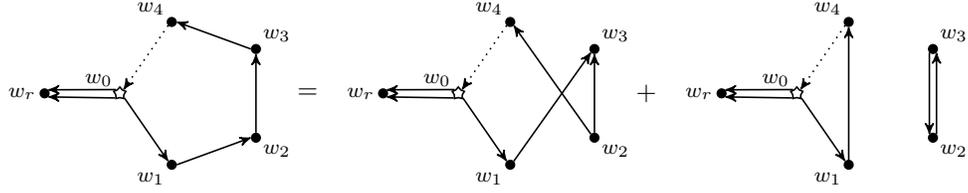
\begin{center}
\begin{tikzpicture}[>=stealth',semithick,inner sep=0.4mm]
\tikzstyle{every label}=[font=\footnotesize]

\node (F) at ( -2,0) [label=left:$w_r$,circle,draw,fill] {};
\node (A) at ( -0.31,-0.95) [label=below left:$w_1$,circle,draw,fill] {};
\node (B) at ( 0.81,-0.59) [label=below right:$w_2$,circle,draw,fill] {};
\node (D) at ( 0.81,0.59) [label=above right:$w_3$,circle,draw,fill] {};
\node (E) at ( -0.31,0.95) [label=above left:$w_4$,circle,draw,fill] {};
\node (H) at (-1,0) [label=above left:$w_0$,star,draw]{};
\draw [->] (H.south east) -- (A.north west);
\draw [->] (A.east) -- (B.west);
\draw [->] (B.north) -- (D.south);
\draw [->] (D.west) -- (E.east);
\draw [->,dotted] (E.south west) -- (H.north east);
\draw [->] (H.south west) -- (F.south east);
\draw [->] (H.north west) -- (F.north east);

\node (a) at (1.5,0) [circle]{$=$};

\node (F') at ( 2.5,0) [label=left:$w_r$,circle,draw,fill] {};
\node (A') at ( -0.31+4.5,-0.95) [label=below left:$w_1$,circle,draw,fill] {};
\node (B') at ( 0.81+4.5,-0.59) [label=below right:$w_2$,circle,draw,fill] {};
\node (D') at ( 0.81+4.5,0.59) [label=above right:$w_3$,circle,draw,fill] {};
\node (E') at ( -0.31+4.5,0.95) [label=above left:$w_4$,circle,draw,fill] {};
\node (H') at (-1+4.5,0) [label=above left:$w_0$,star,draw]{};
\draw [->] (H'.south east) -- (A'.north west);
\draw [->] (A'.north east) -- (D'.south west);
\draw [->] (B'.north) -- (D'.south);
\draw [->] (B'.north west) -- (E'.south east);
\draw [->,dotted] (E'.south west) -- (H'.north east);
\draw [->] (H'.south west) -- (F'.south east);
\draw [->] (H'.north west) -- (F'.north east);

\node (b) at (6,0) [circle]{$+$};

\node (F'') at ( 7,0) [label=left:$w_r$,circle,draw,fill] {};
\node (A'') at ( -0.31+9,-0.95) [label=below left:$w_1$,circle,draw,fill] {};
\node (B'') at ( 0.81+9,-0.59) [label=below right:$w_2$,circle,draw,fill] {};
\node (D'') at ( 0.81+9,0.59) [label=above right:$w_3$,circle,draw,fill] {};
\node (E'') at ( -0.31+9,0.95) [label=above left:$w_4$,circle,draw,fill] {};
\node (H'') at (-1+9,0) [label=above left:$w_0$,star,draw]{};
\draw [->] (A''.north) -- (E''.south);
\draw [->] (B''.north east) -- (D''.south east);
\draw [->] (D''.south west) -- (B''.north west);
\draw [->,dotted] (E''.south west) -- (H''.north east);
\draw [->] (H''.south east) -- (A''.north west);
\draw [->] (H''.south west) -- (F''.south east);
\draw [->] (H''.north west) -- (F''.north east);

\end{tikzpicture}
\end{center}
By induction hypothesis, $B_{1,\ldots,r}H \equiv -B_{1,3,2,4,\ldots,r}H \mod t_2(\ker(t'))$. 

Now, using the same argument as in Step 1, let $P=(w_0,\ldots,w_{r-1})$, $P'=(w_0,w_2,\ldots,w_r)$, and we see that $t_2(PH-P'H)=(B_{1,\ldots,r}H-B_{2,\ldots,r,1}H,0)$, so that $B_{1,\ldots,r}H \equiv B_{2,\ldots,r,1}H \mod t_2(\ker t')$. 

%argument from step 1, with pentagons
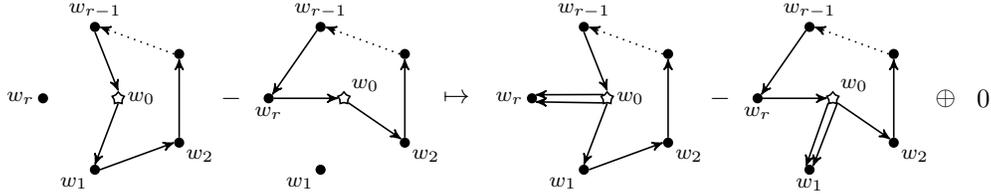
\begin{center}
\begin{tikzpicture}[>=stealth',semithick,inner sep=0.4mm]
\tikzstyle{every label}=[font=\footnotesize]

\node (F) at ( 0,0) [label=right:$w_0$,star,draw] {};
\node (A) at ( -0.31,-0.95) [label=below left:$w_1$,circle,draw,fill] {};
\node (B) at ( 0.81,-0.59) [label=below right:$w_2$,circle,draw,fill] {};
\node (D) at ( 0.81,0.59) [circle,draw,fill] {};
\node (E) at ( -0.31,0.95) [label=above:$w_{r-1}$,circle,draw,fill] {};
\node (H) at (-1,0) [label=left:$w_r$,circle,draw,fill]{};
\draw [->] (A.east) -- (B.west);
\draw [->] (B.north) -- (D.south);
\draw [->,dotted] (D.west) -- (E.east);
\draw [->] (E.south) -- (F.north);
\draw [->] (F.south) -- (A.north);

\node (a) at (1.5,0) [circle] {$-$};

\node (F') at ( 3,0) [label=above right:$w_0$,star,draw] {};
\node (A') at ( -0.31+3,-0.95) [label=below left:$w_1$,circle,draw,fill] {};
\node (B') at ( 0.81+3,-0.59) [label=below right:$w_2$,circle,draw,fill] {};
\node (D') at ( 0.81+3,0.59) [circle,draw,fill] {};
\node (E') at ( -0.31+3,0.95) [label=above:$w_{r-1}$,circle,draw,fill] {};
\node (H') at (-1+3,0) [label=below:$w_r$,circle,draw,fill]{};
\draw [->] (B'.north) -- (D'.south);
\draw [->,dotted] (D'.west) -- (E'.east);
\draw [->] (E'.south west) -- (H'.north east);
\draw [->] (H'.east) -- (F'.west);
\draw [->] (F'.south east) -- (B'.north west);

\node (b) at (4.5,0) [circle] {$\mapsto$};

\node (F'') at ( 6.5,0) [label=right:$w_0$,star,draw] {};
\node (A'') at ( -0.31+6.5,-0.95) [label=below left:$w_1$,circle,draw,fill] {};
\node (B'') at ( 0.81+6.5,-0.59) [label=below right:$w_2$,circle,draw,fill] {};
\node (D'') at ( 0.81+6.5,0.59) [circle,draw,fill] {};
\node (E'') at ( -0.31+6.5,0.95) [label=above:$w_{r-1}$,circle,draw,fill] {};
\node (H'') at (-1+6.5,0) [label=left:$w_r$,circle,draw,fill]{};
\draw [->] (A''.east) -- (B''.west);
\draw [->] (B''.north) -- (D''.south);
\draw [->,dotted] (D''.west) -- (E''.east);
\draw [->] (E''.south) -- (F''.north);
\draw [->] (F''.south) -- (A''.north);
\draw [->] (F''.north west) -- (H''.north east);
\draw [->] (F''.south west) -- (H''.south east);

\node (c) at (8,0) [circle] {$-$};

\node (F''') at ( 9.5,0) [label=above right:$w_0$,star,draw] {};
\node (A''') at ( -0.31+9.5,-0.95) [label=below:$w_1$,circle,draw,fill] {};
\node (B''') at ( 0.81+9.5,-0.59) [label=below right:$w_2$,circle,draw,fill] {};
\node (D''') at ( 0.81+9.5,0.59) [circle,draw,fill] {};
\node (E''') at ( -0.31+9.5,0.95) [label=above:$w_{r-1}$,circle,draw,fill] {};
\node (H''') at (-1+9.5,0) [label=below:$w_r$,circle,draw,fill]{};
\draw [->] (B'''.north) -- (D'''.south);
\draw [->,dotted] (D'''.west) -- (E'''.east);
\draw [->] (E'''.south west) -- (H'''.north east);
\draw [->] (H'''.east) -- (F'''.west);
\draw [->] (F'''.south east) -- (B'''.north west);
\draw [->] (F'''.south west) -- (A'''.north west);
\draw [->] (F'''.south east) -- (A'''.north east);

\node (d) at (11,0) [circle] {$\oplus$};
\node (e) at (11.5,0) [circle]{$0$};

\end{tikzpicture}
\end{center}

We combine all the equivalences above to obtain $$B_{1,\ldots,r}H \equiv B_{2,1,3,\ldots,r}H \equiv -B_{1,2,3,\ldots,r}H\mod t_2(\ker t'),$$ and then $(J,0)\in \im(t_2)$.
\end{step}

\begin{step}
Now let $H \in \mathfrak{F}'_{(0,d_1,\ldots,d_n)}$ be any graph. Then $(J,H)\in \im(t_2)$ for some $J$ containing a subgraph $B$ of the form (\ref{balloon}) from Step 3. Indeed, since $H$ does not admit a $2$-coloring, by Remark \ref{odd cycle or 2-color} it must contain an odd cycle, say $C=(w_1,\ldots,w_r)$ is a subgraph of $H$, $H=CP$ for some $P$. But then $(B_{1,\ldots,r}P,H)=t_2(C'P)$, where $C'$ is the cycle $(w_0,w_1,\ldots,w_{r-1})$. 

%Last argument in Step 4
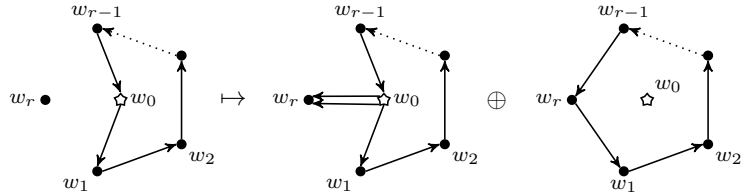
\begin{center}
\begin{tikzpicture}[>=stealth',semithick,inner sep=0.4mm]
\tikzstyle{every label}=[font=\footnotesize]

\node (F) at ( 0,0) [label=right:$w_0$,star,draw] {};
\node (A) at ( -0.31,-0.95) [label=below left:$w_1$,circle,draw,fill] {};
\node (B) at ( 0.81,-0.59) [label=below right:$w_2$,circle,draw,fill] {};
\node (D) at ( 0.81,0.59) [circle,draw,fill] {};
\node (E) at ( -0.31,0.95) [label=above:$w_{r-1}$,circle,draw,fill] {};
\node (H) at (-1,0) [label=left:$w_r$,circle,draw,fill]{};
\draw [->] (A.east) -- (B.west);
\draw [->] (B.north) -- (D.south);
\draw [->,dotted] (D.west) -- (E.east);
\draw [->] (E.south) -- (F.north);
\draw [->] (F.south) -- (A.north);

\node (a) at (1.5,0) [circle] {$\mapsto$};

\node (F') at ( 3.5,0) [label=right:$w_0$,star,draw] {};
\node (A') at ( -0.31+3.5,-0.95) [label=below left:$w_1$,circle,draw,fill] {};
\node (B') at ( 0.81+3.5,-0.59) [label=below right:$w_2$,circle,draw,fill] {};
\node (D') at ( 0.81+3.5,0.59) [circle,draw,fill] {};
\node (E') at ( -0.31+3.5,0.95) [label=above:$w_{r-1}$,circle,draw,fill] {};
\node (H') at (-1+3.5,0) [label=left:$w_r$,circle,draw,fill]{};
\draw [->] (A'.east) -- (B'.west);
\draw [->] (B'.north) -- (D'.south);
\draw [->,dotted] (D'.west) -- (E'.east);
\draw [->] (E'.south) -- (F'.north);
\draw [->] (F'.south) -- (A'.north);
\draw [->] (F'.north west) -- (H'.north east);
\draw [->] (F'.south west) -- (H'.south east);

\node (c) at (5,0) [circle] {$\oplus$};

\node (F'') at ( 7,0) [label=above right:$w_0$,star,draw] {};
\node (A'') at ( -0.31+7,-0.95) [label=below:$w_1$,circle,draw,fill] {};
\node (B'') at ( 0.81+7,-0.59) [label=below right:$w_2$,circle,draw,fill] {};
\node (D'') at ( 0.81+7,0.59) [circle,draw,fill] {};
\node (E'') at ( -0.31+7,0.95) [label=above:$w_{r-1}$,circle,draw,fill] {};
\node (H'') at (-1+7,0) [label=left:$w_r$,circle,draw,fill]{};
\draw [->] (B''.north) -- (D''.south);
\draw [->,dotted] (D''.west) -- (E''.east);
\draw [->] (E''.south west) -- (H''.north east);
\draw [->] (H''.south east) -- (A''.north west);
\draw [->] (A''.east) -- (B''.west);

\end{tikzpicture}
\end{center}
\end{step}
The graph $J=B_{1,\ldots,r}P$ is in $t_2(\ker t')$ by Step 3. Finally, since both $(J,H)$ and $(J,0) \in \im(t_2)$, we obtain $(0,H)\in \im(t_2)$, so this concludes the proof.
\end{proof}

\subsection{Main result}

\begin{comment}
\begin{theorem} \label{main theorem}
Suppose $Y=(\mathbb{P}^1)^n \git_\mathcal{L} PGL_2$, where the polarization $\mathcal{L}=\mathcal{O}_X(d_1,\ldots,d_n)$ has no strictly semi-stable locus. Then $Y$ satisfies Bott vanishing.
\end{theorem}
\end{comment}

Now we can prove the main result.

\begin{proof}[Proof of Theorem \ref{main theorem}]
Note that in our case $PGL_2$ acts freely on $X^{ss}=X^s$, so the GIT quotient $Y$ is indeed a smooth projective variety. If $X^{us}$ has codimension $1$, then we are done by Lemma \ref{is toric} and the fact that every smooth projective toric variety satisfies Bott vanishing (see the references given in \S\ref{introduction section} or Theorem \ref{thm bott for toric}). Otherwise, by Lemma \ref{L'=L} it suffices to show vanishing for $\Omega^j_Y \otimes L$, where $L$ is the descent of the polarization $\mathcal{L}$. If $j=0$, then $H^i(Y,L)=H^i(X,\mathcal{L})^G$ which is certainly $0$ for $i>0$. Assume $j\geq 1$.

From Corollary \ref{can use teleman}, $H^i(Y,\Omega_Y^j \otimes \mathcal{L})=H^i(\mathfrak{X},\Lambda^j L_{\mathfrak{X}} \otimes \mathcal{L})$. This is zero for $i\neq 0,j$ by Corollary \ref{almost bott}. By Lemma \ref{F computes hypercohomology}, we need to show $\mathcal{H}^j(F^\cdot)^G=0$, where $F^\cdot$ is given by (\ref{koszul F}). That is, we need to show that the map
\begin{align*}
(H^0(X,\Omega_X \otimes \mathcal{L}) \otimes S^{j-1}\mathfrak{g}^\vee)^G \overset{t_j}\longrightarrow (H^0(X,\mathcal{L})\otimes S^j\mathfrak{g}^\vee)^G
\end{align*}
is surjective for every $j$. Propositions \ref{graphs for H^1} and \ref{graphs for H^2} show this is true for $j=1$ and $j=2$. Now we do induction on $j$.
Let $j \geq 3$. Consider the short exact sequence from (\ref{tautological on P2}), giving rise to the splitting $S^m \mathfrak{g}^\vee=V_{2m}\oplus S^{m-2}\mathfrak{g}^\vee$ for $m \geq 2$. We use (\ref{tautological on P2}) for $m=j$ and $m=j-1$. Take its pullback to $X\times \mathbb{P}(\mathfrak{g})$ and tensor with the pullbacks of $\Omega_X \otimes \mathcal{L}$ and $\mathcal{L}$, respectively. Then we have a commutative diagram
\[
\begin{tikzcd}[cramped,sep=small]
0 \arrow{r} &(\Omega_X \otimes \mathcal{L})\boxtimes \mathcal{O}_{\mathbb{P}(\mathfrak{g})}(j-3) \arrow{r}\arrow{d} & (\Omega_X \otimes \mathcal{L})\boxtimes \mathcal{O}_{\mathbb{P}(\mathfrak{g})}(j-1) \arrow{r}\arrow{d} &(\Omega_X \otimes \mathcal{L})\boxtimes \mathcal{O}_{\mathbb{P}^1}(2j-2) \arrow{r} &0\phantom{.} \\
0 \arrow{r} &\mathcal{L}\boxtimes \mathcal{O}_{\mathbb{P}(\mathfrak{g})}(j-2) \arrow{r} & \mathcal{L}\boxtimes \mathcal{O}_{\mathbb{P}(\mathfrak{g})}(j) \arrow{r} &\mathcal{L}\boxtimes \mathcal{O}_{\mathbb{P}^1}(2j) \arrow{r} &0.
\end{tikzcd}
\]

The vertical maps are given by the section $s \in H^0(X \times \mathbb{P}(\mathfrak{g}),T_X \boxtimes \mathfrak{g}^\vee)$ defining the map $\Omega_X \rightarrow \mathfrak{g}^\vee$, as usual. Taking global sections we see that the map $H^0(X,\Omega_X \otimes \mathcal{L}) \otimes S^{j-1}\mathfrak{g}^\vee \overset{d_j} \longrightarrow H^0(X,\mathcal{L})\otimes S^j\mathfrak{g}^\vee$ splits as
\[
\begin{tikzcd}[cramped,sep=scriptsize]
H^0(X,\Omega_X \otimes \mathcal{L})\otimes V_{2(j-1)} \arrow{d}{t}\arrow{rrd}{t'} &\oplus& H^0(X,\Omega_X \otimes \mathcal{L})\otimes S^{j-3}\mathfrak{g}^\vee \arrow{d}{t_{j-2}} \\
H^0(X,\mathcal{L})\otimes V_{2j} &\oplus& H^0(X,\mathcal{L})\otimes S^{j-2}\mathfrak{g}^\vee.
\end{tikzcd}
\]

By induction hypothesis, the restriction $t_{j-2}^G$ of $t_{j-2}$ to invariant sections is surjective, while the restriction $t^G$ is surjective by Corollary \ref{exactness on V2j}. As a consequence, $t_j^G=(t^G,t'^G+t_{j-2}^G)$ is surjective. This completes the proof.
\end{proof}

\section{A note on Fano varieties}

As mentioned in the Introduction, (non-toric) Fano varieties satisfying Bott vanishing are particularly interesting. We have $T_Y=\Omega_Y^{n-1}\otimes K_Y^{\vee}$, so if $K_Y^\vee$ is ample and $Y$ satisfies Bott vanishing, then $H^1(Y,T_Y)$ must be zero, and $Y$ must be rigid. In particular, Bott vanishing holds for at most finitely many smooth complex Fano varieties in each dimension.

If $Y=(\mathbb{P}^1)^n\git PGL_2$ is as in Theorem \ref{main theorem} and it is non-toric, then $\Pic Y$ is the $G$-ample cone of $(\mathbb{P}^1)^n$ (see the proof of Lemma \ref{L'=L}) and $K_Y$ is the descent of $K_X=\mathcal{O}_X(-2,\ldots,-2)$. If $Y$ is Fano, then by Lemma \ref{L'=L} it has to be the quotient $(\mathbb{P}^1)^n\git_{\mathcal{O}_X(2,\ldots,2)}PGL_2$. Observe that $\mathcal{O}_X(2,\ldots,2)$ has no strictly semi-stable locus if and only if $n$ is odd. In other words, Theorem \ref{main theorem} provides us with exactly one non-toric example of a Fano variety satisfying Bott vanishing in each even dimension. In the case of dimension $2$, this was the quintic del Pezzo surface.

An interesting non-example comes from a Fano threefold that contains the quintic del Pezzo surface as a hyperplane section. Let $M$ be the Fano threefold over $\mathbb{C}$ of index $2$ and degree $5$, with Picard number $1$. The canonical line bundle is $K_M=\mathcal{O}_M(-2)$, where $\mathcal{O}_M(1)$ is the ample generator of the Picard group. $M$ is a rigid Fano threefold, isomorphic to a linear section of the Grassmannian $\Gr(2,5)\subset \mathbb{P}^9$ by a subspace $\mathbb{P}^6\subset\mathbb{P}^9$. The quintic del Pezzo surface $V$ can be realized as a divisor in the linear system $|\mathcal{O}_M(1)|$. It can be computed that the Hodge numbers of $M$ are $h^{0,0}(M)=h^{1,1}(M)=1$ and zero otherwise, in particular $h^{1,2}(M)=0$. The description of $M$ can be found in \cite[\S 5.1]{kuznetsov-fano3} or \cite[\S 4]{mukai-fano3}.

This variety $M$ does not satisfy Bott vanishing. Indeed, we claim that $H^1(M,\Omega_M^2(1))$ has dimension at least $3$. Observe that by Serre duality, this is the same as saying that $H^2(M,\Omega_M(-1))$ has dimension $\geq 3$. To show this, we follow a strategy similar to \cite[Lemma 1.2]{jahnke-fano3}, using the dualized tangent sequence
\begin{align}\label{seq on V}
0\rightarrow \mathcal{O}_V(-1)\rightarrow\Omega_M|_V\rightarrow\Omega_V\rightarrow 0
\end{align}
and the ideal sequence tensored with $\Omega_M$
\begin{align}\label{seq on M}
0\rightarrow\Omega_M(-1)\rightarrow\Omega_M\rightarrow\Omega_M|_V\rightarrow 0.
\end{align}

By the Kodaira-Akizuki-Nakano vanishing theorem \cite[Theorem 4.2.3]{positivity1}, we have the vanishing $H^1(M,\Omega_M(-1))=0$. Using the fact that $h^{0,0}(M)=h^{1,1}(M)=1$, $h^{1,2}(M)=0$ and sequence (\ref{seq on M}), we get an exact sequence
\begin{align*}
0\rightarrow H^1(M,\Omega_M)\rightarrow H^1(V,\Omega_M|_V)\rightarrow H^2(M,\Omega_M(-1))\rightarrow 0,
\end{align*}
so $h^2(\Omega_M(-1))=h^1(\Omega_M|_V)-1$. It suffices to check that $h^1(\Omega_M|_V)\geq4$.

Now take sequence $(\ref{seq on V})$ and observe $\mathcal{O}_V(-1)=K_V$ by adjunction. Since $h^{2,1}(V)=h^{1,2}(V)=0$, we get
\begin{align*}
0\rightarrow H^1(V,\Omega_M|_V)\rightarrow H^1(V,\Omega_V)\rightarrow H^2(V,K_V)\rightarrow H^2(V,\Omega_M|_V)\rightarrow 0.
\end{align*}
From the Hodge numbers of $V$, $h^{1,1}(V)=10-5=5$, $h^{2,2}(V)=1$, we see
\begin{align*}
h^2(\Omega_M|_V)+5=h^1(\Omega_M|_V)+1.
\end{align*}
In particular, $h^1(\Omega_M|_V)\geq 4$, proving the claim and the fact that $M$ does not satisfy Bott vanishing.

%\begin{thebibliography}
%\bibliographystyle{amsalpha}
%\bibliography{biblio}
%\end{thebibliography}

\end{document}